\documentclass[a4paper,11pt]{article} 
 \usepackage[italian,english]{babel}			
 \usepackage{amsmath,amssymb,amsthm}
 \usepackage{amsfonts} 
 \usepackage{graphicx} 
 \usepackage{braket}
 \usepackage{times}				
 \usepackage{ifthen}
 \usepackage{xspace}
 \usepackage{fancyhdr}
 \usepackage{mathrsfs}
 \usepackage{enumerate}
 \usepackage{verbatim}
 \usepackage{hyperref}
 \usepackage[latin1]{inputenc}
\usepackage{color} 
 \definecolor{Red}{cmyk}{0,1,1,0.2}

  \def\d{b_\Omega}
 
 \def\oh{\overline{\gamma}}
 \def\L{{\rm Lip}}
 \def \t{\overline{t}_k}
 \def \lo{\mathcal{P}_{m_0}^{\rm Lip}(\Gamma)}
 \def \hl{H_L}

  \usepackage{fancyhdr}
 
 \usepackage{fancyhdr}
 \newtheorem{definition}{Definition}[section] 
 \theoremstyle{definition}
 
 \theoremstyle{remark}
 \newtheorem{remark}{Remark}[section]
  \theoremstyle{plain}
 \newtheorem{theorem}{Theorem}[section]
 \newtheorem{lemma}{Lemma}[section]
 \newtheorem{proposition}{Proposition}[section]
 \newtheorem{corollario}{Corollary}[section]
 
 \numberwithin{equation}{section}
 \allowdisplaybreaks

 \title{\bf $\bf{C^{1,1}}$--smoothness of constrained solutions in the calculus of variations with application to mean field games\footnote{This work was partly supported by the University of Rome Tor Vergata (Consolidate the Foundations 2015) and by the Istituto Nazionale di Alta Matematica ``F. Severi'' (GNAMPA 2016 Research Projects). The authors acknowledge the MIUR Excellence Department Project awarded to the Department of Mathematics, University of Rome Tor Vergata, CUP E83C18000100006. The second author is grateful to the Università Italo Francese (Vinci Project 2015)}}
 \author{{\sc Piermarco Cannarsa}\thanks{%
 		Dipartimento di Matematica, Universit\`{a} di Roma ``Tor Vergata" - {\tt cannarsa@mat.uniroma2.it}}\\
 	{\sc Rossana Capuani}\thanks{%
 		 		Dipartimento di Matematica, Universit\`{a} di Roma ``Tor Vergata" and CEREMADE, Universit\'e Paris-Dauphine - {\tt capuani@mat.uniroma2.it}}\\ {\sc Pierre Cardaliaguet}\thanks{%
 		 		CEREMADE, Universit\'e Paris-Dauphine - {\tt cardaliaguet@ceremade.dauphine.fr}}
 		 	\\}
 \date{}
\begin{document}
 	\maketitle
 \begin{abstract}
 \noindent
We derive necessary optimality conditions for minimizers of regular functionals in the calculus of variations under smooth state constraints. In the literature, this classical problem is widely investigated. The novelty of our result lies in the fact that the presence of state constraints enters the Euler-Lagrange equations as a local feedback, which allows to derive the $C^{1,1}$-smoothness of solutions. As an application, we discuss a constrained Mean Field Games problem, for which our optimality conditions allow to construct Lipschitz relaxed solutions, thus improving the existence result in [Cannarsa, P., Capuani, R., \textsl{Existence and uniqueness for Mean Field Games with state constraints}, http://arxiv.org/abs/1711.01063].
 \end{abstract}
 \section{Introduction}
 
  The centrality of necessary conditions in optimal control  is well-known and has originated an immense literature in the fields of optimization and nonsmooth analysis, see, e.g., \cite{aa}, \cite{clarke}, \cite{dubovitskii1964extremum}, \cite{lr}, \cite{rv}, \cite{3v}.
 
 In control theory, the celebrated Pontryagin Maximum Principle plays the role of the classical Euler-Lagrange equations in the calculus of variations. In the case of unrestricted state space, such conditions provide Lagrange multipliers in the form of so-called co-states, that is, solutions to a suitable adjoint system which satisfy a certain transversality condition. Among various applications of necessary optimality conditions is the deduction of  further regularity properties for minimizers which, a priori, would just be absolutely continuous.
 
 When state constraints are present, a large body of results provide adaptations of the Pontryagin Principle by introducing appropriate corrections in the adjoint system. The price to pay for such extensions usually consists in a reduced regularity of optimal trajectories which, due to constraint reactions, turn out to be just Lipschitz continuous with associated co-states of bounded variation,  see \cite{fran}.
 
 The maximum principle under state constraints was first established by Dubovitskii and Milyutin \cite{dubovitskii1964extremum} (see also the monograph \cite{3v} for different forms of maximum principle under state constraints). It may happen that the maximum principle is degenerate and does not yield much information (abnormal maximum principle). As explained in  \cite{bettiol2007normality, bettiol2016normality, frankowska2006regularity, frankowska2009normality} in various contexts, the so-called ``inward pointing condition" generally ensures the normality of the maximum principle under state constraints. In our setting (calculus of variation problem, with constraints on positions but not on velocities), this will never be an issue. The maximum principle under state constraints generally involves an adjoint state which is the sum of a $W^{1,1}$ map and a map of bounded variation. This latter mapping may be very irregular and have infinitely many jumps \cite{milyutin2000certain}, which allows for discontinuities in optimal controls. However, under suitable assumptions (requiring regularity of the data and the affine dynamics with respect to controls), it has been shown that optimal controls and the corresponding adjoint state are continuous, and even Lipschitz continuous: see the seminal work by Hager \cite{hager1979lipschitz} (in the convex setting) and the subsequent contributions by Malanowski \cite{malanowski1978regularity} and Galbraith and Vinter \cite{galbraith2003lipschitz} (in much more general frameworks). Generalization to less smooth frameworks can also be found in \cite{bettiol2008holder, frankowska2006regularity}.\\
 Let $\Omega\subset \mathbb{R}^n $ be a bounded open domain with $C^2$ boundary.
 Let $\Gamma$ be the metric subspace of $AC(0,T;\mathbb{R}^n)$ defined by
 \begin{equation*}
 \Gamma=\Big\{\gamma\in AC(0,T;\mathbb{R}^n):\ \gamma(t)\in\overline{\Omega},\ \ \forall t\in[0,T]\Big\},
 \end{equation*}
 with the uniform metric.
 For any $x\in\overline{\Omega}$, we set
 \begin{equation*}
 \Gamma[x]=\left\{\gamma\in\Gamma: \gamma(0)=x\right\}.
 \end{equation*}
We  consider the problem of minimizing the classical functional of the calculus of variations
 \begin{equation*}
  J[\gamma]=\int_{0}^T f(t,\gamma(t),\dot{\gamma}(t)) \,dt + g(\gamma(T)).
 \end{equation*}
Let $U\subset \mathbb{R}^n$ be an open set such that $\overline{\Omega}\subset U$.
Given $ x\in \overline \Omega$, we consider the constrained minimization problem 
\begin{equation}\label{M1}
\inf_{\gamma\in\Gamma[ x]}J[\gamma],  \ \ \ \ \mbox{where}  \ \ \ \ \
J[\gamma]=\Big\{\int_{0}^T f(t,\gamma(t),\dot{\gamma}(t)) \,dt + g(\gamma(T))\Big\},
\end{equation} 
where $f: [0,T]\times U\times \mathbb{R}^n\rightarrow \mathbb{R}$ and $g:U\rightarrow \mathbb{R}$ .
In this paper, we obtain a certain formulation of the necessary optimality conditions for the above problem, which are particularly useful to study the regularity of minimizers.
More precisely, given a minimizer $\gamma^\star \in \Gamma[x]$ of \eqref{M1}, we prove that
there exists a Lipschitz continuous arc $p:[0,T]\rightarrow \mathbb{R}^n$ such that 
\begin{align}\label{sis}
\begin{cases}
\dot{\gamma^\star}(t)=-D_pH(t,\gamma^\star(t),p(t)) \ \ \ \ & \text{for all} \ t \in [0,T]\\
\dot{p}(t)=D_xH(t,\gamma^\star(t),p(t))-\Lambda(t,\gamma^\star,p)D\d(\gamma^\star(t)) &\text{for a. e.} \ t \in [0,T]\\
\end{cases}
\end{align}
where $\Lambda$ is a bounded continuous function independent of $\gamma^\star$ and $p$ (Theorem \ref{51}). By the above necessary conditions we derive a sort of maximal regularity, showing that any solutions $\gamma^\star$ is of class $C^{1,1}$. As is customary in this kind of problems, the proof relies on the analysis of suitable penalized functional which has the following form:
\begin{equation*}
\inf_{\tiny\begin{array}{c}
	\gamma\in AC(0,T;\mathbb{R}^n)\\
	\gamma(0)=x
	\end{array}}  \left\{\int_{0}^T \Big[f(t,\gamma(t),\dot{\gamma}(t)) +\frac{1}{\epsilon}\ d_{{\Omega}}(\gamma(t))\Big]\,dt+ \frac{1}{\delta}\ d_{{\Omega}}(\gamma(T))+ g(\gamma(T))\right\}.
\end{equation*}
Then, we show that the solutions of the penalized problem remain in $\overline{\Omega}$ (Lemma\ref{lemma5}).\\
A direct consequence of our necessary conditions is the Lipschitz regularity of the value function associated to \eqref{M1} (Proposition \ref{fw}).\\
Our interest is also motivated by application to mean field games, as we explain below. Mean field games (MFG) theory has been developed simultaneously by Lasry and Lions (\cite{8}, \cite{9}, \cite{10}) and by Huang, Malham\'{e} and Caines (\cite{h1}, \cite{h2}) in order to study  differential games with an infinite number of rational players in competition. The simplest MFG model leads to systems of partial differential equations involving two unknown functions: the value function $u$ of an optimal control problem of a typical player and the density $m$ of population of players. In the presence of state constraints, the usual construction of solutions to the MFG system has to be completely revised because the minimizers of the problem lack many of the good properties of the unconstrained case. Such constructions are discussed in detail in \cite{cc}, where a relaxed notion of solution to the constrained MFG problem was introduced following the so-called Lagrangian formulation (see \cite{bb}, \cite{bc}, \cite{bcs}, \cite{b}, \cite{pc}, \cite{cms}).
In this paper, applying our necessary conditions, we deduce the existence of more regular solutions than those constructed in \cite{cc}, assuming the data be Lipschitz continuous.\\
This paper is organised as follows. In Section 2 we introduce the notation and recall preliminary results. In Section 3, under suitable assumptions we prove the necessary conditions for constrained problem. Moreover, $C^{1,1}$-smoothness of minimizers of our problem is derived. In Section 4, we apply our necessary conditions to prove the Lipschitz regularity of the value function for constrained problem. Furthermore, we deduce the existence of more regular constrained MFG equilibria. Finally, in Appendix we prove a technical result on limiting subdifferential.
  \section{Preliminaries}
  Throughout this paper we denote by $|\cdot|$ and $\langle  \cdot  \rangle$ , respectively, the Euclidean norm and scalar product in $\mathbb{R}^n$. Let $A\in\mathbb{R}^{n\times n}$ be a matrix. We denote by $||\cdot||$ the norm of $A$ defined as follows
  $$
  ||A||=\max_{x\in \mathbb{R}^n, \; |x|=1} ||Ax|| \ .
  $$
 For any subset $S \subset \mathbb{R}^n$, $\overline{S}$ stands for its closure, $\partial S$ for its boundary, and $S^c$ for $\mathbb{R}^n\setminus S$. We denote by $\mathbf{1}_{S}:\mathbb{R}^n\rightarrow \{0,1\}$ the characteristic function of $S$, i.e.,
  \begin{align*}
  \mathbf{1}_{S}(x)=
  \begin{cases}
  1  \ \ \ &x\in S,\\
  0 &x\in S^c.
  \end{cases}
  \end{align*}
  We write $AC(0,T;\mathbb{R}^n)$ for the space of all absolutely continuous $\mathbb{R}^n$-valued functions on $[0,T]$, equipped with the uniform norm $||\gamma||_\infty ={\rm sup}_{[0,T]}\ |\gamma(t)|$. We observe that $AC(0,T;\mathbb{R}^n)$ is not a Banach space.\\
  Let $U$ be an open subset of $\mathbb{R}^n$. $C(U)$ is the space of all continuous functions on $U$ and $C_b(U)$ is the space of all bounded continuous functions on $U$. $C^k(U)$ is the space of all functions $\phi:U\rightarrow\mathbb{R}$ that are k-times continuously differentiable. Let $\phi\in C^1(U)$. The gradient vector of $\phi$ is denoted by $D\phi=(D_{x_1}\phi,\cdots , D_{x_n}\phi)$, where $D_{x_i}\phi =\frac{\partial \phi}{\partial x_i}$. Let $\phi \in C^k(U)$ and let $\alpha=(\alpha_1,\cdots,\alpha_n) \in \mathbb{N}^n$ be a multiindex. We define $D^\alpha \phi=D^{\alpha_1}_{x_1}\cdots D^{\alpha_n}_{x_n}\phi$.
$C^k_b(U)$ is the space of all function $\phi\in C^k(U)$ and such that
  $$
  \|\phi\|_{k,\infty}:=\sup_{\tiny\begin{array}{c}
  	x\in U\\ 
  	|\alpha|\leq k
  	\end{array}} |D^\alpha\phi(x)|<\infty
  $$
  Let $\Omega$ be a bounded open subset of $\mathbb{R}^n$ with $C^2$ boundary.  $C^{1,1}(\overline{\Omega})$ is the space of all the functions $C^1$ in a neighborhood $U$ of $\Omega$ and with locally Lipschitz continuous first order derivates in $U$.\\
  The distance function from  $\overline{\Omega}$ is the function $d_\Omega :\mathbb{R}^n \rightarrow [0,+ \infty[$ defined by
  \begin{equation*}
  d_\Omega(x):= \inf_{y \in \overline{\Omega}} |x-y| \ \ \ \ \ (x\in \mathbb{R}^n).
  \end{equation*}
  We define the oriented boundary distance from $\partial \Omega$ by
  \begin{equation*}
  \d(x)=d_\Omega(x) -d_{\Omega^c}(x) \ \ \ \ (x\in\mathbb{R}^n).
  \end{equation*}
  We recall that, since the boundary of $\Omega$ is of class $C^2$, there exists $\rho_0>0$ such that
  \begin{equation}\label{dn}
  \d(\cdot)\in C^2_b \ \ \text{on} \ \ \Sigma_{\rho_0}=\Big\{y\in B(x,\rho_0): x\in \partial\Omega\Big\}.
  \end{equation}
  Throughout the paper, we suppose that $\rho_0$ is fixed so that \eqref{dn} holds.\\
  Take a continuous function $f: \mathbb{R}^n\rightarrow \mathbb{R}$ and a point $x\in \mathbb{R}^n$. A vector $p\in\mathbb{R}^n$ is said to be a proximal subgradient of $f$ at $x$ if there exists $\epsilon>0$ and $C\geq 0$ such that
  \begin{equation*}
  p\cdot(y-x)\leq f(y)-f(x) + C|y-x|^2 \ \ \mbox{for all}\ y \ \mbox{that satisfy} \ |y-x|\leq \epsilon.
  \end{equation*}
  The set of all proximal subgradients of $f$ at $x$ is called the proximal subdifferential of $f$ at $x$ and is denoted by $\partial^p f(x)$.
  A vector $p\in\mathbb{R}^n$ is said to be a limiting subgradient of $f$ at $x$ if there exist sequences $x_i\in \mathbb{R}^n$, $p_i\in \partial^p f(x_i)$ such that $x_i\rightarrow x$ and $p_i\rightarrow p$ ($i\rightarrow \infty$).\\
  The set of all limiting subgradients of $f$ at $x$ is called the limiting subdifferential and is denoted by $\partial f(x)$.
  In particular, for the distance function we have the following result.
  \begin{lemma}\label{lemmaapp}
  	Let $\Omega$ be a bounded open subset of $\mathbb{R}^n$ with $C^2$ boundary. Then, for every $x\in \mathbb{R}^n$ it holds
  	\begin{align*}
  	\partial^p d_\Omega(x)=\partial d_\Omega(x)=
  	\begin{cases}
  	D\d(x) \ \ \ \ \ &0<\d(x)<\rho_0,\\
  	D\d(x)[0,1] &x\in\partial\Omega,\\
  	0 &x\in\Omega,
  	\end{cases}
  	\end{align*}
  where $\rho_0$ is as in \eqref{dn}.
  \end{lemma}
  \noindent
  The proof is given in the Appendix.\\
  Let $X$ be a separable metric space. $C_b(X)$ is the space of all bounded continuous functions on $X$. We denote by $\mathscr{B}(X)$ the family of the Borel subset of $X$ and by $\mathcal{P}(X)$ the family of all Borel probability measures on $X$.
  The support of $\eta \in \mathcal{P}(X)$, $supp(\eta)$, is the closed set defined by
  \begin{equation*}
  supp (\eta) := \Big \{x \in X: \eta(V)>0\ \mbox{for each neighborhood V of $x$}\Big\}.
  \end{equation*}
  We say that a sequence $(\eta_i)\subset \mathcal{P}(X)$ is narrowly convergent to $\eta \in \mathcal{P}(X)$ if
  \begin{equation*}
  \lim_{i\rightarrow \infty} \int_X f(x)\,d\eta_i(x)=\int_X f(x) \,d\eta \ \ \ \ \forall f \in C_b(X).
  \end{equation*}
  We denote by $d_1$ the Kantorovich-Rubinstein distance on $X$, which---when $X$ is compact---can be characterized as follows 
  \begin{equation}\label{2.7}
  d_1(m,m')=sup\Big\{\int_X f(x)\,dm(x)-\int_X f(x)\,dm'(x)\ \Big|\ f:X\rightarrow\mathbb{R}\ \ \mbox{is 1-Lipschitz} \Big\}, 
  \end{equation}
  for all $m, m'\in\mathcal{P}(X)$.\\
  Let $\Omega$ be a bounded open subset of $\mathbb{R}^n$ with $C^2$ boundary. We write $\L(0,T;\mathcal{P}(\overline{\Omega}))$ for the space of all maps $m:[0,T]\rightarrow \mathcal{P}(\overline{\Omega})$ that are Lipschitz continuous with respect to $d_1$, i.e.,
  \begin{equation}\label{lipm}
  d_1(m(t),m(s))\leq C|t-s|, \ \ \ \ \forall t,s\in[0,T],
  \end{equation}
  for some constant $C\geq 0$. We denote by $\L(m)$ the smallest constant that verifies \eqref{lipm}.\\
\section{Necessary conditions and smoothness of minimizers}\label{sec3}
 \subsection{Assumptions and main result}
Let $\Omega \subset \mathbb{R}^n$ be a bounded open set with $C^2$ boundary. Let $\Gamma$ be the metric subspace of $AC(0,T;\mathbb{R}^n)$ defined by
  \begin{equation*}
  \Gamma=\Big\{\gamma\in AC(0,T;\mathbb{R}^n):\ \gamma(t)\in\overline{\Omega},\ \ \forall t\in[0,T]\Big\}.
  \end{equation*}
  For any $x\in\overline{\Omega}$, we set
  \begin{equation*}
\Gamma[x]=\left\{\gamma\in\Gamma: \gamma(0)=x\right\}.
  \end{equation*}
 Let $U\subset \mathbb{R}^n$ be an open set such that $\overline{\Omega}\subset U$.
 Given $ x\in \overline \Omega$, we consider the constrained minimization problem 
 \begin{equation}\label{M}
\inf_{\gamma\in\Gamma[ x]}J[\gamma],  \ \ \ \ \mbox{where}  \ \ \ \ \
  J[\gamma]=\Big\{\int_{0}^T f(t,\gamma(t),\dot{\gamma}(t)) \,dt + g(\gamma(T))\Big\}.
 \end{equation}
 We denote by $\mathcal{X}[x]$ the set of solutions of \eqref{M}, that is
 \begin{equation*}
 \mathcal{X}[x]=\Big\{\gamma^\star\in \Gamma[x]: J[\gamma^\star]=\inf_{\Gamma[x]}J[\gamma]\Big\}.
 \end{equation*}
 We assume that $f: [0,T]\times U\times \mathbb{R}^n\rightarrow \mathbb{R}$ and $g:U\rightarrow \mathbb{R}$ satisfy the following conditions.
 \begin{enumerate}
 	\item[(g1)] $g\in C^1_b(U)$
 	\item[(f0)] $f\in C\big([0,T]\times U\times \mathbb{R}^n\big)$ and for all $t\in[0,T]$ the function  $(x,v)\longmapsto f(t,x,v)$ is differentiable. Moreover, $D_xf$, $D_vf$ are continuous on $[0,T]\times U\times \mathbb{R}^n$ 
 	and there exists a constant $M\geq 0$ such that
 	\begin{equation}
 	|f(t,x,0)|+|D_xf(t,x,0)|+|D_vf(t,x,0)|\leq M \ \ \ \ \forall\ (t,x)\in [0,T]\times U.\label{bm}
 	\end{equation}
 	\item[(f1)] For all $t\in[0,T]$ the map $(x,v)\longmapsto D_{v}f(t,x,v)$ is continuously differentiable and there exists a constant $\mu\geq 1$ such that
 	\begin{align}
 	&\frac{I}{\mu} \leq D^2_{vv}f(t,x,v)\leq I\mu,\label{f2}\\
 	&||D_{vx}^2f(t,x,v)||\leq \mu(1+|v|), \label{fvx}
 	\end{align}
 	for all $(t,x,v)\in [0,T]\times U\times \mathbb{R}^n$, where $I$ denotes the identity matrix.
 	\item[(f2)] For all $(x,v)\in U\times\mathbb{R}^n$ the function $t\longmapsto f(t,x,v)$ and the map $t\longmapsto D_vf(t,x,v)$ are Lipschitz continuous. Moreover, there exists a constant $\kappa\geq 0$ such that
 	\begin{align}
 	&|f(t,x,v)-f(s,x,v)|\leq \kappa(1+|v|^2)|t-s|\label{lf1}\\
 	&|D_vf(t,x,v)-D_vf(s,x,v)|\leq \kappa(1+|v|)|t-s|\label{fvt}
 	\end{align} 
 	for all $t$, $s\in [0,T]$, $x\in  U$, $v\in\mathbb{R}^n$.
 \end{enumerate}
 \begin{remark}
 	By classical results in the calculus of variation (see, e.g., \cite[Theorem 11.1i]{c}), there exists at least one minimizer of \eqref{M} in $\Gamma$ for any fixed point $x\in\overline{\Omega}$.
 \end{remark}
 \noindent
 In the next lemma we show that (f0)-(f2) imply the useful growth conditions for $f$ and for its derivatives.
  \begin{lemma}\label{r3}
 	Suppose that (f0)-(f2) hold. Then, there exists a positive constant $C(\mu,M)$ depending only on $\mu$ and $M$ such that
 	\begin{align}
 	&|D_vf(t,x,v)|\leq C(\mu,M)(1+|v|)\label{l4},\\
 	&|D_xf(t,x,v)|\leq C(\mu,M)(1+|v|^2)\label{lx},\\
 	&\frac{1}{4\mu}|v|^2-C(\mu,M)\leq f(t,x,v)\leq 4\mu|v|^2 +C(\mu,M)\label{l3},
 	\end{align}
 	for all $(t,x,v)\in [0,T]\times U\times \mathbb{R}^n$.
 \end{lemma}
 \begin{proof}
 	By \eqref{bm}, and by \eqref{f2} one has that
 	\begin{align*}
 	|D_v f(t,x,v)|&\leq |D_vf(t,x,v)-D_vf(t,x,0)| +|D_vf(t,x,0)|\\
 	&\leq\int_0^1 \big| D_{vv}^2f(t,x,\tau v)\big||v|\,d\tau+ |D_vf(t,x,0)|\leq \mu|v|+M\leq C(\mu,M)(1+|v|)
 	\end{align*}
 	and so \eqref{l4} holds. Furthermore, by \eqref{bm}, and by \eqref{fvx} we have that
 	\begin{align*}
 	&|D_xf(t,x,v)|\leq |D_xf(t,x,v)-D_xf(t,x,0)|+|D_xf(t,x,0)|\leq \int_0^1 \big|D_{xv}^2f(t,x,\tau v)\big||v| \,d \tau +M\\
 	&\leq \mu(1+|v|)|v|+M \leq C(\mu,M)(1+|v|^2).
 	\end{align*}
 	Therefore, \eqref{lx} holds. Moreover, fixed $v\in \mathbb{R}^n$ there exists a point $\xi$ of the segment with endpoints $0$, $v$ such that
 	\begin{equation*}
 	f(t,x,v)=f(t,x,0) + \langle D_vf(t,x,0),v\rangle +\frac{1}{2} \langle D_{vv}^2f(t,x,\xi)v,v\rangle .
 	\end{equation*} 
 	By \eqref{bm}, \eqref{f2}, and by \eqref{l4} we have that
 	\begin{align*}
 	&	-C(\mu,M)+\frac{1}{4\mu}|v|^2\leq-M-C(\mu,M)|v|+\frac{1}{2\mu}|v|^2\leq f(t,x,v)\leq M+C(\mu,M)|v|+\frac{\mu}{2}|v|^2\\
 	&\leq C(\mu,M)+4\mu|v|^2,
 	\end{align*}  
 	and so \eqref{l3} holds. This completes the proof.
 \end{proof}
\noindent
 In the next result we show a special property of the minimizers of \eqref{M}.
 \begin{lemma}\label{lb1}
 	For any $x\in\overline{\Omega}$ and for any $\gamma^\star\in\mathcal{X}[x]$ we have that
 	\begin{equation*}
 	\int_0^T\frac{1}{4\mu}|\dot{\gamma}^\star(t)|^2\,dt\leq K,
 	\end{equation*}
 	where 
 	\begin{equation}\label{k} K:=T\Big(C(\mu,M)+M\Big)+2\max_{U}|g(x)|.
 	\end{equation}
 \end{lemma}
 \begin{proof}
 	Let $x\in\overline{\Omega}$ and let $\gamma^\star\in \mathcal{X}[x]$. By comparing the cost of $\gamma^\star$ with the cost of the constant trajectory $\gamma^\star(t)\equiv x$, one has that
 	\begin{align}\label{2}
 	&\int_0^T f(t,\gamma^\star(t), \dot{\gamma}^\star(t)) \, dt + g(\gamma^\star(T))\leq \int _0^T f(t,x,0) \,dt + g(x)\\
 	&\leq T\max_{[0,T]\times U}|f(t,x,0)|+\max_{U}|g(x)|.\nonumber
 	\end{align}
 	Using \eqref{bm} and \eqref{l3} in (\ref{2}), one has that
 	\begin{equation*}
 	\int_0^T\frac{1}{4\mu}|\dot{\gamma}^\star(t)|^2\,dt\leq K,
 	\end{equation*}
 	where 
 	\begin{equation*} K:=T\Big(C(\mu,M)+M\Big)+2\max_{U}|g(x)|.
 	\end{equation*}
 \end{proof}
\noindent
We denote by $H:[0,T]\times U\times\mathbb{R}^n \rightarrow \mathbb{R}$ the Hamiltonian 
 \begin{equation*}
 H(t,x,p)=\sup_{v\in \mathbb{R}^n} \Big\{ -\langle p,v\rangle - f(t,x,v)\Big\},\qquad \forall \ (t,x,p)\in [0,T]\times U\times \mathbb{R}^n.
 \end{equation*}
Our assumptions on $f$ imply that $H$ satisfies the following conditions.
 	\begin{enumerate}
 	\item[(H0)] $H\in C\big([0,T]\times U\times \mathbb{R}^n\big)$ and for all $t\in[0,T]$ the function  $(x,p)\longmapsto H(t,x,p)$ is differentiable.  Moreover, $D_xH$, $D_pH$ are continuous on $[0,T]\times U\times \mathbb{R}^n$ and there exists a constant $M'\geq 0$ such that
 		\begin{equation}
 		|H(t,x,0)|+|D_xH(t,x,0)|+|D_pH(t,x,0)|\leq M' \ \ \ \ \forall\ (t,x)\in [0,T]\times U.
 		\end{equation}
 	 \item[(H1)] For all $t\in[0,T]$ the map $(x,p)\longmapsto D_{p}H(t,x,p)$ is continuously differentiable and 
 	 \begin{align}
 	 &\frac{I}{\mu} \leq D^2_{pp}H(t,x,p)\leq I\mu,\label{h1b}\\
 	 &||D_{px}^2H(t,x,p)||\leq C(\mu,M')(1+|p|), \label{40}
 	 \end{align}
 	 for all $(t,x,p)\in [0,T]\times U\times \mathbb{R}^n$, where $\mu$ is the constant given in (f1) and $C(\mu,M')$ depends only on $\mu$ and $M'$.
 	  \item[(H2)] For all $(x,p)\in U\times\mathbb{R}^n$ the function $t\longmapsto H(t,x,p)$ and the map $t\longmapsto D_pH(t,x,p)$ are Lipschitz continuous. Moreover 
 	 \begin{align}
 	 &|H(t,x,p)-H(s,x,p)|\leq \kappa C(\mu,M')(1+|p|^2)|t-s|,\\
 	 &|D_pH(t,x,p)-D_pH(s,x,p)|\leq \kappa C(\mu,M')(1+|p|)|t-s|,\label{41}
 	 \end{align} 
 	 for all $t$, $s\in [0,T]$, $x\in  U$, $p\in\mathbb{R}^n$, where $\kappa$ is the constant given in (f2) and $C(\mu,M')$ depends only on $\mu$ and $M'$.
 	\end{enumerate}
 	\begin{remark}
 	Arguing as in Lemma \ref{r3} we deduce that
 	\begin{align}
 &|D_pH(t,x,p)|\leq C(\mu,M')(1+|p|),\label{312}\\
  &|D_xH(t,x,p)|\leq C(\mu,M')(1+|p|^2),\label{39}\\
   &	\frac{1}{4\mu}|p|^2-C(\mu,M')\leq H(t,x,p)\leq 4\mu|p|^2  +C(\mu,M'),\label{c0}
\end{align}
for all $(t,x,p)\in[0,T]\times U\times\mathbb{R}^n$ and $C(\mu,M')$ depends only on $\mu$ and $M'$.
 	\end{remark}
\noindent	
Under the above assumptions on $\Omega$, $f$ and $g$ our necessary conditions can be stated as follows.
 \begin{theorem}\label{51}
For any $x\in\overline{\Omega}$ and any $\gamma^\star \in\mathcal{X}[x]$ the following holds true.
 	\begin{enumerate}
 	\item[(i)] $\gamma^\star$ is of class $C^{1,1}([0,T];\overline{\Omega})$.
 \item[(ii)] There exist:
 	\begin{enumerate}
 	\item[(a)] a Lipschitz continuous arc $p:[0,T]\rightarrow \mathbb{R}^n$, 
 	\item[(b)]a constant $\nu\in\mathbb{R}$ such that
 	\begin{equation*}
 	0\leq\nu\leq \max\left\{1,2\mu \ \sup_{x\in U}\Big|D_pH(T,x,Dg(x))\Big|\right\},
 	\end{equation*}
 	 \end{enumerate}
 	 which satisfy the adjoint system
 	\begin{align}\label{sr}
 	\begin{cases}
 	\dot{\gamma}^\star=-D_pH(t,\gamma^\star,p) \ \ \ &\mbox{for all} \ t\in[0,T],\\
 	\dot p=D_xH(t,\gamma^\star,p)-\Lambda(t,\gamma^\star,p)\mathbf{1}_{\partial{\Omega}}(\gamma^\star) D\d(\gamma^\star) &\mbox{for a.e.}\ t\in[0,T],
 	\end{cases}
 	\end{align}
 	and the transversality condition
 	\begin{equation*}
 p(T)= D g(\gamma^\star(T))+ \nu D\d(\gamma^\star(T))\mathbf{1}_{\partial\Omega}(\gamma^\star(T)),
 	\end{equation*}
 	where $\Lambda:[0,T]\times\Sigma_{\rho_0}\times\mathbb{R}^n\rightarrow \mathbb{R}^n$ is a bounded continuous function independent of $\gamma^\star$ and $p$.
 	\end{enumerate}
 	Moreover, 
 	\begin{enumerate}
 	\item[(iii)] the following estimate holds
 	\begin{equation}\label{lstar}
 	||\dot{\gamma}^\star||_\infty\leq L^\star, \ \ \ \forall \gamma^\star\in \mathcal{X}[x],
 	\end{equation}
where $L^\star=L^\star(\mu,M',M,\kappa,T,||Dg||_\infty,||g||_\infty)$.
 	\end{enumerate}
 \end{theorem}
 \noindent
The (feedback) function $\Lambda$ in \eqref{sr} can be computed explicitly, see Remark \ref{lambda1} below.
 \subsection{Proof of Theorem \ref{51} for $U=\mathbb{R}^n$}
 In this section we prove Theorem \ref{51} in the special case of $U=\mathbb{R}^n$. The proof for a general open set $U$ will be given in the next section.\\
The idea of proof is based on \cite[ Theorem 2.1]{ccc} where the Maximum Principle under state constraints is proved for a Mayer problem. The reasoning requires several intermediate steps.\\
Fix $x\in \overline{\Omega}$. The key point is to approximate the constrained problem by penalized problems as follows
 \begin{equation}\label{v1}
 \inf_{\tiny\begin{array}{c}
 	\gamma\in AC(0,T;\mathbb{R}^n)\\
 	\gamma(0)=x
 	\end{array}}  \left\{\int_{0}^T \Big[f(t,\gamma(t),\dot{\gamma}(t)) +\frac{1}{\epsilon}\ d_{{\Omega}}(\gamma(t))\Big]\,dt+ \frac{1}{\delta}\ d_{{\Omega}}(\gamma(T))+ g(\gamma(T))\right\}.
 \end{equation}
 Then, we will show that, for $\epsilon>0$ and $\delta\in (0,1]$ small enough, the solutions of the penalized problem remain in $\overline{\Omega}$.\\
 Observe that the Hamiltonian associated with the penalized problem is given by
 \begin{equation}\label{perth}
 	H_\epsilon(t,x,p) =
\sup_{v\in \mathbb{R}^n} \Big\{ -\langle p,v\rangle - f(t,x,v)-\frac{1}{\epsilon}\ d_{{\Omega}}(x)\Big\}  
 =  H(t,x,p)-\frac{1}{\epsilon}\ d_{{\Omega}}(x),
 \end{equation}
 for all $(t,x,p)\in [0,T]\times \mathbb{R}^n\times \mathbb{R}^n$.\\
 By classical results in the calculus of variation (see, e.g., \cite[Section 11.2]{c}), there exists at least one mimimizer of \eqref{v1} in $AC(0,T;\mathbb{R}^n)$ for any fixed initial point $ x\in\overline{\Omega}$. We denote by $\mathcal{X}_{\epsilon,\delta}[x]$ the set of solutions of (\ref{v1}).
\begin{remark}\label{r4}
 Arguing as in Lemma \ref{lb1} we have that, for any $x\in \overline{\Omega}$, all $\gamma\in\mathcal{X}_{\epsilon,\delta}[x]$ satisfy
 \begin{equation}
 \int_0^T\Big[\frac{1}{4\mu}|\dot{\gamma}(t)|^2+\frac{1}{\epsilon}\ d_{{\Omega}}(\gamma(t))\Big] \,dt\leq K,
 \end{equation}
 where $K$ is the constant given in \eqref{k}.
 \end{remark}
 \noindent
 The first step of the proof consists in showing that the solutions  of the penalized problem remain in a neighborhood of $\overline{\Omega}$. 
 \begin{lemma}\label{lemma3.1} Let $\rho_0$ be such that \eqref{dn} holds. For any $\rho\in(0,\rho_0]$, there exists $\epsilon(\rho)>0$ such that for all $\epsilon\in (0,\epsilon(\rho)]$ and all $\delta\in(0,1]$ we have that  
 	\begin{equation}\label{dist}
 	\forall \ x\in\overline{\Omega},\ \gamma\in\mathcal{X}_{\epsilon,\delta}[ x] \ \ \Longrightarrow\ \ \sup_{t\in[0,T]}d_{\Omega}(\gamma(t))\leq \rho. 
 	\end{equation}
 \end{lemma}
 \begin{proof} 
We argue by contradiction. Assume that, for some $\rho>0$, there exist sequences $\{\epsilon_k\}$, $\{\delta_k\}$, $\{t_k\}$, $\{x_k\}$ and $\{\gamma_k\}$ such that
 	\begin{equation*}
\epsilon_k \downarrow 0,\ \delta_k>0,\ t_k\in[0,T],\ x_k\in \overline{\Omega},\ \gamma_k\in \mathcal{X}_{\epsilon_k,\delta_k}[x_k]\  \mbox{and}\ d_\Omega(\gamma_k(t_k))> \rho,\  \ \ \mbox{for all} \ k\geq 1.
 	\end{equation*}	
By Remark \ref{r4}, one has that for all $k\geq1$
 	\begin{eqnarray*}
 	\int_{0}^T \Big[\frac{1}{4\mu} |\dot{\gamma}_k(t)|^2 +\frac{1}{\epsilon_k}\ d_{\Omega}(\gamma_k(t))\Big]\,dt\leq K, 
 	\end{eqnarray*}
 	where $K$ is the constant given in \eqref{k}. The above inequality  implies that $\gamma_k$ is $1/2-$H\"{o}lder continuous with H\"{o}lder constant $(4\mu K)^{1/2}$. Then, by the Lipschitz continuity of $d_{\Omega}$ and the regularity of $\gamma_k$, we have that 
 	\begin{equation*}
 	d_\Omega(\gamma_k(t_k))-d_\Omega(\gamma_k(s))\leq (4\mu K)^{1/2}|t_k-s|^{1/2}, \ \ s\in[0,T].
 	\end{equation*}
 	Since $d_{ \Omega}(\gamma_k(t_k))> \rho$, one has that
 	\begin{equation*}
 	d_\Omega(\gamma_k(s))>\rho-(4\mu K)^{1/2}|t_k-s|^{1/2}.
 	\end{equation*}
 	Hence, $d_{\Omega}(\gamma_k(s))\geq \rho/2$ for all $s\in J:=[t_k-\frac{\rho^2}{16\mu K}, t_k+ \frac{\rho^2}{16\mu K}]\cap [0,T]$ and all $k\geq 1$. So, 
 	\begin{align*}
 	K\geq \frac{1}{\epsilon_k}\int_0^T d_\Omega(\gamma_k(t))\,dt\geq \frac{1}{\epsilon_k}\int_{J} d_\Omega(\gamma_k(t))\,dt \geq \frac{1}{\epsilon_k} \frac{\rho^3}{32\mu K}.
 	\end{align*}
 	But the above inequality contradicts the fact that $\epsilon_k\downarrow 0$. So, \eqref{dist} holds true.
 \end{proof}
 \noindent
 In the next lemma, we show the necessary conditions for the minimizers of the penalized problem.
 \begin{lemma}\label{lemma.1.2}
 	 Let $\rho\in(0,\rho_0]$ and let $\epsilon\in(0,\epsilon(\rho)]$, where $\epsilon(\rho)$ is given by Lemma \ref{lemma3.1}. Fix $\delta\in(0,1]$, let $ x_0\in\overline{\Omega}$, and let $\gamma\in\mathcal{X}_{\epsilon,\delta}[ x_0]$. Then, 
 	 \begin{enumerate}
 	 \item[(i)] $\gamma$ is of class $C^{1,1}([0,T];\mathbb{R}^n)$;
 	 \item[(ii)] there exists an arc $p\in \L(0,T;\mathbb{R}^n)$, a measurable map $\lambda:[0,T]\rightarrow [0,1]$, and a constant $\beta\in [0,1]$ such that
 	\begin{equation}\label{s3}
 	\begin{cases}				
 	\dot{\gamma}(t)=-D_pH(t,\gamma(t),p(t)), \ \ \ &\mbox{for all}\ t\in[0,T],\\
 	\dot{p}(t) =  D_x H(t,\gamma(t),p(t))-\frac{\lambda(t)}{\epsilon}\  D\d(\gamma(t)), &\mbox{for a.e.}\ t\in[0,T],\\
 	p(T)=Dg(\gamma(T))+\frac{\beta}{\delta}\ D\d(\gamma(T)),
 	\end{cases}
 	\end{equation}
 	where
 	\begin{equation}\label{lambda}
 	\lambda(t) \in 
 	\begin{cases}
 	\{0\}  & {\rm if }\; \gamma(t)\in \Omega,\\ 
 	\{1\} & {\rm if }\; 0<d_\Omega(\gamma(t))< \rho, \\
 	[0,1]  & {\rm if }\; \gamma(t) \in \partial \Omega,
 	\end{cases}
 	\end{equation}
 	and
 	\begin{equation}\label{beta}
 	\beta \in 
 	\begin{cases}
 	\{0\}  & {\rm if }\; \gamma(T)\in \Omega,\\ 
 	\{1\} & {\rm if }\; 0<d_\Omega(\gamma(T))< \rho, \\
 	[0,1]  & {\rm if }\; \gamma(T) \in \partial \Omega.
 	\end{cases}
 	\end{equation}
 	\end{enumerate}	
 	Moreover,
 	\begin{enumerate}
 	\item[(iii)] the function
 	\begin{equation*}
 	r(t):=H(t,\gamma(t),p(t))-\frac{1}{\epsilon} \ d_{\Omega}(\gamma(t)),  \ \ \ \forall t \in [0,T]
 	\end{equation*}
 	belongs to $AC(0,T;\mathbb{R})$ and satisfies 
 	$$
 	\int_0^T|\dot{r}(t)|\,dt\leq \kappa(T+4\mu K),
 	$$
 	where $K$ is the constant given in \eqref{k} and $\mu$, $\kappa$ are the constants in \eqref{lf1} and \eqref{l3}, respectively;
 	\item[(iv)] the following estimate holds
 	\begin{equation}\label{65}
 	|p(t)|^2\leq 4\mu\left[\frac{1}{\epsilon}d_\Omega(\gamma(t))+ \frac{C_1}{\delta^2}\right],\ \ \ \ \ \forall t \in[0,T],
 	\end{equation}
 	where $C_1=8\mu+8\mu||Dg||_\infty^2+2C(\mu,M')+ \kappa(T+4\mu K)$.
 	\end{enumerate}
 \end{lemma}
 \begin{proof}
 	In order to use the Maximum Principle in the version of \cite[ Theorem 8.7.1]{3v}, we rewrite (\ref{v1}) as a Mayer problem in a higher dimensional state space. Define $X(t)\in \mathbb{R}^n\times \mathbb{R}$ as
 	\begin{align*}
 	X(t)=
 	\begin{pmatrix}
 	\gamma(t) \\ z(t)
 	\end{pmatrix},
 	\end{align*}
 	where $z(t)=\int_0^t \big[f(s,\gamma(s),\dot \gamma(s)) +\frac{1}{\epsilon}\ d_{\Omega}(\gamma(s))\big]\, ds$. Then the state equation becomes 
 	\begin{align*}
 	\begin{cases}
 	\dot{X}(t)=
 	\begin{pmatrix}
 	\dot{\gamma}(t) \\ \dot{z}(t)
 	\end{pmatrix}=\mathcal{F}_\epsilon(t,X(t),u(t)),\\
 	\\
 	X(0)=
 	\begin{pmatrix}
 	 x_0 \\ 0
 	\end{pmatrix}.
 	\end{cases}
 	\end{align*}
 	where
 	\begin{align*}
 	\mathcal{F}_\epsilon(t,X,u)=\begin{pmatrix}
 	u \\ \mathcal{L}_{\epsilon}(t,x,u)
 	\end{pmatrix}
 	\end{align*}
 	and $\mathcal{L}_{\epsilon}(t,x,u)=f(t,x,u)+\frac{1}{\epsilon}\ d_{{\Omega}}(x)$ for $X=(x,z)$ and $(t,x,z,u)\in [0,T]\times \mathbb{R}^n\times \mathbb{R} \times \mathbb{R}^n$. Thus, (\ref{v1}) can be written as
 	\begin{equation}\label{b1}
 	\min \Big\{ \Phi(X^u(T)):u \in L^1\Big\},
 	\end{equation}
 	where $\Phi(X)=g(x) + \frac{1}{\delta}\ d_{{\Omega}}(x) + z$ for any $X=(x,z)\in \mathbb{R}^n\times \mathbb{R}$. The associated unmaximized Hamiltonian is given by
 	\begin{equation*}
 	\mathcal{H}_\epsilon (t,X,P,u)=- \langle P, \mathcal{F}_\epsilon(t,X,u)\rangle, \qquad \forall (t,X,P,u)\in [0,T]\times  \mathbb{R}^{n+1} \times  \mathbb{R}^{n+1}\times  \mathbb{R}^n. 
 	\end{equation*}
 	We observe that, as $\gamma(\cdot)$ is minimizer for (\ref{v1}), $X$ is minimizer for (\ref{b1}).
 	Hence, the hypotheses of \cite[Theorem 8.7.1]{3v} are satisfied. It follows that there exist $P(\cdot)=(p(\cdot), b(\cdot)) \in AC(0,T;\mathbb{R}^{n+1})$, $r(\cdot) \in AC(0,T;\mathbb{R})$, and $\lambda_0\geq 0$ such that
 	\begin{enumerate}
 		\item[(i)] $\big(P,\lambda_0\big)\not \equiv \big(0,0\big)$,
 		\item[(ii)] $\big(\dot r(t), \dot{P}(t)\big)\in co \ \partial_{t,X}\mathcal{H}_\epsilon\big(t,X(t),P(t),\dot{\gamma}(t)\big)$, a.e $t\in[0,T]$,
 		\item[(iii)] $P(T)\in \lambda_0 \partial \Phi(X^u(T))$,
 		\item[(iv)] $\mathcal{H}_\epsilon \big(t,X(t),P(t),\dot{\gamma}(t)\big)=\max_{u\in  \mathbb{R}^n} \mathcal{H}_\epsilon \big(t,X(t),P(t),u\big)$, a.e. $t\in [0,T]$,
 		\item[(v)]$\mathcal{H}_\epsilon\big(t,X(t),P(t),\dot{\gamma}(t)\big)=r(t)$, a.e. $t\in [0,T]$,			
 	\end{enumerate}
 	where  $\partial_{t,X}\mathcal{H}_\epsilon$ and $\partial\Phi$ denote the limiting subdifferential of $\mathcal{H}_\epsilon$ and $\Phi$ with respect to $(t,X)$ and $X$ respectively, while $co$ stands for the closed convex hull. Using the definition of $\mathcal{H}_\epsilon$ we have that
 	\begin{align}
 	(p,b,\lambda_0)&\not \equiv (0,0,0), \label{514}\\
 	(\dot{r}(t), \dot{p}(t))& \in -b(t)\ co\ \partial_{t,x}\mathcal{L}_{\epsilon}(t,\gamma(t),\dot{\gamma}(t)), \label{510}\\
 	\dot{b}(t)&=0, \label{511}\\
 	p(T) &\in \lambda_0\ \partial(g+\frac{1}{\delta} \ d_{\Omega})(\gamma(T)), \label{512}\\
 	b(T)&=\lambda_0,\label{513} \\
 	r(t) & =H_\epsilon(t,\gamma(t),p(t)),
 	\end{align}
 	where $\partial_{t,x}\mathcal{L}_{\epsilon}$ and  $\partial(g+\frac{1}{\delta}\  d_{\Omega})$ stands for the limiting subdifferential of $\mathcal{L}_{\epsilon}(\cdot,\cdot,u)$ and $g(\cdot)+\frac{1}{\delta} d_{\Omega}(\cdot)$. 
 	We claim that $\lambda_0>0$. Indeed, suppose that $\lambda_0=0$. Then $b\equiv 0$ by \eqref{511} and \eqref{513}. Moreover, $p(T)=0$ by \eqref{512}. It follows from \eqref{510} that $p\equiv 0$, which is in contradiction with (\ref{514}).
 	So, $\lambda_0> 0$ and we may rescale $p$ and $b$ so that $b(t)=\lambda_0=1$ for any $t\in [0,T]$.\\
 	Note that the Weierstrass Condition (iv) becomes
 	\begin{equation}\label{wc}
 	-\langle p(t), \dot{\gamma}(t)\rangle-f(t,\gamma(t),\dot \gamma(t))= \sup_{u\in \mathbb{R}^n} \Big\{-\langle p(t),u\rangle -f(t,\gamma(t),u)\Big\}.
 	\end{equation}
 	Therefore 
 	\begin{equation}\label{gh}
 	\dot{\gamma}(t)=-D_pH(t,\gamma(t),p(t)),\qquad {\rm a.e.}\; t\in [0,T]. 
 	\end{equation}
 By Lemma \ref{lemmaapp}, by the definition of $\rho$, and by \eqref{lf1} we have that 
 	\begin{align*}
 	\partial_{t,x} {\mathcal L}_{\epsilon} (t,x,u)\subset
 	\begin{cases}
 	[-\kappa(1+|u|^2), \kappa(1+|u|^2)]\times D_x f(t,x,u) & {\rm if }\; x\in \Omega,\\ 
 	[-\kappa(1+|u|^2), \kappa(1+|u|^2)] \times \big(D_x f(t,x,u)+ \frac{1}{\epsilon}\ D\d(x)\big) & {\rm if }\; 0<\d(x)< \rho, \\
 	[-\kappa(1+|u|^2), \kappa(1+|u|^2)]\times \big(D_x f(t,x,u)+ \frac{1}{\epsilon}[0,1] \ D\d(x)\big) & {\rm if }\; x\in \partial \Omega.
 	\end{cases}
 \end{align*}
 	Thus \eqref{510} implies that there exists $\lambda(t)\in [0,1]$  as in \eqref{lambda} such that 
 	\begin{align}				
 	|\dot{r}(t)| & \leq \kappa(1+|\dot{\gamma}(t)|^2), \ \ \forall t \in [0,T], &\label{rt} \\ 
 	\dot{p}(t) & =  -D_x f(t,\gamma(t),\dot{\gamma}(t))-\frac{\lambda(t)}{\epsilon} \ D\d(\gamma(t)),\ \; \text{a.e.}\ t\in [0,T]. & \label{398bis}
 	\end{align}
 	Hence, by \eqref{rt}, and by Remark \ref{r4} we conclude that
 	\begin{equation}\label{drt}
 	\int_0^T|\dot{r}(t)|\,dt\leq \kappa\int_0^T (1+|\dot{\gamma}(t)|^2)\,dt\leq \kappa(T+4\mu K).
 	\end{equation}
 	Moreover, by Lemma \ref{lemmaapp}, and by assumption on $g$, one has that
 	$$
 	\partial\Big(g+\frac{1}{\delta}\ d_{\Omega}\Big)(x) \subset 
 	\begin{cases}
 	Dg(x)  & {\rm if }\; x\in \Omega,\\ 
 	Dg(x)+\frac{1}{\delta}\ D\d(x) & {\rm if }\; 0<\d(x)< \rho, \\
 	Dg(x)+\frac{1}{\delta}[0,1]\ D\d(x)  & {\rm if }\; x \in \partial \Omega.
 	\end{cases}
 	$$
 	So, by \eqref{512}, there exists $\beta\in  [0,1]$ as in \eqref{beta} such that
 	\begin{equation}\label{fbg}
 	p(T)= Dg(x)+\frac{\beta}{\delta} \ D\d(x).
 	\end{equation}
 	Finally, by well-known properties of the Legendre transform one has that  
 	\begin{equation*}
 	D_xH(t,x,p)=  -D_xf\big(t,x, - D_pH(t,x,p)\big).
 	\end{equation*}
 	So, recalling \eqref{gh}, \eqref{398bis} can be rewritten as
 	$$
 	\dot{p}(t) =  D_x H(t,\gamma(t),p(t))-\frac{\lambda(t)}{\epsilon}\  D\d(\gamma(t)), \; \text{a.e.}\ t\in [0,T].
 	$$
 We have to prove estimate \eqref{65}. Recalling \eqref{perth} and \eqref{c0}, we have that
 \begin{align*}
 H_\epsilon(t,\gamma(t),p(t))=H(t,\gamma(t),p(t))-\frac{1}{\epsilon}\ d_{\Omega}(\gamma(t))\geq \frac{1}{4\mu}|p(t)|^2-C(\mu,M')-\frac{1}{\epsilon}\ d_{\Omega}(\gamma(t)).
 \end{align*}
So, using \eqref{drt} one has that 
 \begin{equation*}
 |H_\epsilon(T,\gamma(T),p(T))-H_\epsilon(t,\gamma(t),p(t))|=|r(T)-r(t)|\leq \int_t^T|\dot{r}(s)|\,ds\leq \kappa(T+4\mu K).
 \end{equation*}
 Moreover, \eqref{fbg} implies that $|p(T)|\leq \frac{1}{\delta}+||Dg||_\infty$. Therefore, using again \eqref{c0}, we obtain 
 \begin{align*}
 &\frac{1}{4\mu}|p(t)|^2-C(\mu,M')-\frac{1}{\epsilon}\ d_{\Omega}(\gamma(t))\leq  H_\epsilon(t,\gamma(t),p(t))\leq H_\epsilon(T,\gamma(T),p(T)) +\kappa(T+4\mu K)\\
 &\leq  4\mu|p(T)|^2+C(\mu,M') +\kappa(T+4\mu K)\leq 8\mu\left[ \frac{1}{\delta^2}+||Dg||_\infty^2\right]+C(\mu,M') +\kappa(T+4\mu K).
 \end{align*}
 Hence,
 \begin{equation*}
 |p(t)|^2\leq 4\mu\left[\frac{1}{\epsilon} d_{\Omega}(\gamma(t))+\frac{C_1}{\delta^2}\right],
 \end{equation*}
 where $C_1=8\mu+8\mu||Dg||_\infty^2+2C(\mu,M')+\kappa(T+4\mu K)$. This completes the proof of \eqref{65}.\\
 Finally, by the regularity of $H$, we have that $p\in \L(0,T;\mathbb{R}^n)$. So, $\gamma\in C^{1,1}([0,T];\mathbb{R}^n)$. Observing that the right-hand side of the equality $\dot{\gamma}(t)=-D_pH(t,\gamma(t),p(t))$ is continuous we conclude that this equality holds for all $t$ in $[0,T]$.
 \end{proof}
 \noindent
 \begin{lemma}\label{gc2}
		Let $\rho\in(0,\rho_0]$ and let $\epsilon\in(0, \epsilon(\rho)]$, where $ \epsilon(\rho)$ is given by Lemma \ref{lemma3.1}. Fix $\delta\in(0,1]$, let $ x\in\overline{\Omega}$, and let $\gamma\in\mathcal{X}_{\epsilon,\delta}[ x]$. If $\gamma(\overline{t})\notin \partial\Omega$ for some $\overline{t}\in [0,T]$, then there exists $\tau>0$ such that $\gamma\in C^2\left(\left(\overline{t}-\tau,\overline{t}+\tau\right)\cap [0,T];\mathbb{R}^n\right)$. 
				\end{lemma}
		\begin{proof}
			Let $\gamma\in\mathcal{X}_{\epsilon,\delta}[x]$ and let $\overline{t}\in [0,T]$ be such that $\gamma(\overline{t})\in \Omega \cup(\mathbb{R}^n\setminus \overline{\Omega})$. If $\gamma(\overline{t})\in\mathbb{R}^n\setminus\overline{\Omega}$, then there exists $\tau>0$ such that $\gamma(t)\in \mathbb{R}^n\setminus\overline{\Omega}$ for all $t\in I:=(\overline{t}-\tau,\overline{t}+\tau)\cap[0,T]$. By Lemma \ref{lemma.1.2}, we have that there exists $p\in \L(0,T;\mathbb{R}^n)$ such that
			\begin{align*}
			&\dot{\gamma}(t)=-D_pH(t,\gamma(t),p(t)),\\
			&\dot{p}(t)= D_xH(t,\gamma(t),p(t))-\frac{1}{\epsilon}D\d(\gamma(t)),
			\end{align*}
			for $t\in I$. Since $p(t)$ is Lipschitz continuous for $t\in I$, and $\dot{\gamma}(t)=-D_pH(t,\gamma(t),p(t))$, then $\gamma$ belongs to $C^1\left(I;\mathbb{R}^n\right)$. Moreover, by the regularity of $H$, $\d$, $p$, and $\gamma$ one has that $\dot{p}(t)$ is continuous for $t\in I$. Then $p\in C^1\left(I;\mathbb{R}^n\right)$. Hence, $\dot{\gamma}\in C^1\left(I;\mathbb{R}^n\right)$.
			So, $\gamma\in C^2\left(I;\mathbb{R}^n\right)$. Finally, if $\gamma(\overline{t})\in\Omega$, the conclusion follows by a similar argument. 
\end{proof}
 \noindent
 In the next two lemmas, we show that, for $\epsilon>0$ and $\delta\in(0,1]$ small enough, any solution $\gamma$ of problem \eqref{v1} belongs to $\overline{\Omega}$ for all $t\in[0,T]$. For this we first establish that, if $\delta\in(0,1]$ is small enough and $\gamma(T)\notin \overline{\Omega}$, then the function $t\longmapsto \d(\gamma(t))$ has nonpositive slope at $t=T$. Then we prove that the entire trajectory $\gamma$ remains in $\overline{\Omega}$ provided $\epsilon$ is small enough. Hereafter, we set
 \begin{equation*}
 \epsilon_0=\epsilon(\rho_0),\ \ \ \mbox{where} \ \rho_0 \ \mbox{is such that \eqref{dn} holds and} \ \epsilon(\cdot) \ \mbox{is given by Lemma \ref{lemma3.1}}.
 \end{equation*}
 \begin{lemma}\label{lemma4}
 	Let 
 	\begin{equation}\label{deltao}
 	\delta= \frac{1}{2\mu N}\wedge 1,
 	\end{equation}
 	where 
 	\begin{equation*}
 	N=\sup_{x\in \mathbb{R}^n}|D_pH(T,x,Dg(x))|.
 	\end{equation*}
 	Fix any $\delta_1 \in (0,\delta]$ and let $ x\in \overline{\Omega}$. Let $\epsilon\in(0, \epsilon_0]$. If $\gamma\in\mathcal{X}_{\delta_1,\epsilon}[ x]$ is such that $\gamma(T)\notin\overline{\Omega}$, then
 \begin{equation*}
 	\langle \dot{\gamma}(T),D\d(\gamma(T))\rangle\leq 0.
 	\end{equation*}
 \end{lemma}
  \begin{proof}
 	As $\gamma(T) \notin \overline{\Omega}$, by Lemma \ref{lemma.1.2} we have that $p (T)=Dg(\gamma(T))+\frac{1}{\delta}\ D\d(\gamma(T))$. Hence,
 	\begin{align*}
 	\Big\langle &D_pH\big(T,\gamma(T),p(T)\big),D\d(\gamma(T))\Big\rangle=\Big\langle D_pH\big(T,\gamma(T),Dg(\gamma(T))\big),D\d(\gamma(T))\Big\rangle\\
 	&+\Big\langle D_pH\big(T,\gamma(T),Dg(\gamma(T))+\frac{1}{\delta}\ D\d(\gamma(T))\big)-D_pH\big(T,\gamma(T),Dg(\gamma(T))\big),D\d(\gamma(T))\Big\rangle.
 	\end{align*}
 	Recalling that $D^2_{pp}H(t,x,p)\geq \frac{I}{\mu}$, one has that
 	\begin{align*}
 	\Big\langle &D_pH\big(T,\gamma(T),Dg(\gamma(T))+\frac{1}{\delta}\ D\d(\gamma(T))\big)-D_pH\big(T,\gamma(T),Dg(\gamma(T))\big),\frac{1}{\delta}\ D\d(\gamma(T))\Big\rangle \\
 	&\geq \frac{1}{2\mu} \frac{1}{\delta^2}\ |D\d(\gamma(T))|^2= \frac{1}{2\delta^{2}\mu}.
 	\end{align*}
 	So,
 	\begin{equation*}
 	\Big\langle D_pH\big(T,\gamma(T),p(T)\big),D\d(\gamma(T))\Big\rangle\geq \frac{1}{2\delta\mu} -|D_pH\big(T,\gamma(T),Dg(\gamma(T))\big)|.
 	\end{equation*}
 	Therefore, we obtain
 	\begin{align*}
 	\big\langle \dot{\gamma}(T),D\d(\gamma(T))\big\rangle & =-\Big\langle D_pH\big(T,\gamma(T),p(T)),D\d(\gamma(T)\big)\Big\rangle \\ & \leq -\frac{1}{2\delta\mu} +|D_pH(T,\gamma(T),Dg(\gamma(T)))|.
 	\end{align*}
 	Thus, choosing $\delta$ as in \eqref{deltao} gives the result. 
 \end{proof}
 \begin{lemma}\label{lemma5}
 	Fix $\delta$ as in \eqref{deltao}. Then there exists $\epsilon_1\in(0,\epsilon_0]$, such that for any $\epsilon\in(0,\epsilon_1]$
 	\begin{equation*}
 	\forall x\in \overline{\Omega}, \ \gamma\in\mathcal{X}_{\epsilon,\delta}[x] \ \ \Longrightarrow \ \ \gamma(t)\in \overline{\Omega} \ \ \ \forall t\in[0,T].
 	\end{equation*}
 \end{lemma}
 \begin{proof}
We argue by contradiction. Assume that there exist sequences $\{\epsilon_k\}$, $\{t_k\}$, $\{x_k\}$, $\{\gamma_k\}$ such that
\begin{equation}\label{contra}
\epsilon_k \downarrow 0, \ t_k \in [0,T], \ x_k \in \overline{\Omega}, \ \gamma_k\in\mathcal{X}_{\epsilon_k,\delta}[x_k] \ \mbox{and} \ \gamma_k(t_k) \notin \overline{\Omega}, \ \ \ \mbox{for all}\  k\geq 1.
\end{equation} 
Then, for each $k\geq 1$ one could find an interval with end-points $0\leq a_k <b_k\leq T$ such that
 \begin{equation*}
 \begin{cases}
 d_\Omega(\gamma_k(a_k))=0,\\
 d_\Omega(\gamma_k(t))>0 \ \ \ t\in(a_k,b_k),\\
 d_\Omega(\gamma_k(b_k))=0 \ \ \mbox{or else} \ \ b_k=T.
 \end{cases}
 \end{equation*}
Let $\overline{t}_k\in(a_k,b_k]$ be such that 
\begin{equation*}
d_\Omega(\gamma_k(\overline{t}_k))=\max_{t\in[a_k,b_k]} d_\Omega(\gamma_k(t)).
\end{equation*}
We note that, by Lemma \ref{gc2}, $\gamma_k$ is of class $C^2$ in a neighborhood of $\t$.\\
\underline{Step 1} \\
We claim that 
\begin{equation}\label{dss}
\frac{d^2}{dt^2}d_\Omega(\gamma_k(t))\Big|_{t=\overline{t}_k}\leq 0.
\end{equation}
Indeed, \eqref{dss} is trivial if $\overline{t}_k\in(a_k,b_k)$. Suppose $\overline{t}_k=b_k$. Since $\overline{t}_k$ is a maximum point of the map $t\longmapsto d_\Omega(\gamma_k(t))$ and $\gamma_k(\overline{t}_k)\notin\overline{\Omega}$, we have that $d_\Omega(\gamma_k(\overline{t}_k))\neq 0$. So, $b_k=T=\overline{t}_k$ and we get
\begin{equation*}
\frac{d}{dt} d_\Omega(\gamma_k(t))\Big|_{t=\overline{t}_k}\geq 0.
\end{equation*}
Moreover, Lemma \ref{lemma4} yields
\begin{equation*}
\frac{d}{dt} d_\Omega(\gamma_k(t))\Big|_{t=\overline{t}_k}\leq0.
\end{equation*}
So,
\begin{equation*}\label{dds1}
\frac{d}{dt} d_\Omega(\gamma_k(t))\Big|_{t=\overline{t}_k}=0,
\end{equation*}
and we have that \eqref{dss} holds true at $\overline{t}_k=T$.\\
\underline{Step 2}\\
 Now, we prove that
\begin{equation}\label{mu}
 \frac{1}{\mu\epsilon_k} \leq C(\mu,M',\kappa)\left[1+4\mu\frac{ C_1}{\delta^2}+\frac{4\mu}{\epsilon_k} \ d_{{\Omega}}(\gamma_k(\overline{t}_k))\right],  \  \ \ \ \forall k\geq 1,
\end{equation}
where $C_1=8\mu+8\mu||Dg||_\infty^2+2C(\mu,M')+\kappa(T+4\mu K)$ and the constant $C(\mu,M',\kappa)$ depends only on $\mu$, $M'$ and $\kappa$. Indeed, since $\gamma$ is of class $C^2$ in a neighborhood of $\overline{t}_k$ one has that
\begin{align}\label{dg2}
\ddot{\gamma}(\overline{t}_k)=&-D_{pt}^2H(\t,\gamma(\t),p(\t)) -\left\langle D_{px}^2H(\t,\gamma(\t),p(\t)),\dot{\gamma}(\t)\right\rangle \\
&-\left\langle D_{pp}^2H(\t,\gamma(\t),p(\t)),\dot{p}(\t)\right\rangle.\nonumber
\end{align}
Developing the second order derivative of $d_\Omega\circ\gamma$, by \eqref{dg2} and the expression of the derivatives of $\gamma$ and $p$ in Lemma \ref{lemma.1.2} one has that	
\begin{eqnarray*}
	0 &\geq &
	\left\langle D^2d_\Omega(\gamma (\t))\dot \gamma (\t), \dot \gamma (\t) \right\rangle  
	+
	\left\langle Dd_\Omega(\gamma (\t)),\ddot \gamma (\t)\right\rangle  
	\\
	&=& \left \langle D^2d_\Omega(\gamma(\t))D_{p}H (\t,\gamma(\t), p(\t)), D_{p}H(\t,\gamma(\t), p( \t))\right\rangle  
	\\
	& &   -\left\langle Dd_\Omega(\gamma (\t)), D_{pt}^2H (\t,\gamma (\t), p (\t))\right\rangle \\
	& & +
	\left\langle Dd_\Omega(\gamma (\t)), D_{px}^2H(\t,\gamma (\t), p(\t)) D_pH (\t,\gamma (\t), p (\t))\right\rangle \\ 
	& & -\left\langle Dd_\Omega(\gamma (\t)), D_{pp}^2H (\t,\gamma (\t), p(\t)) D_xH (\t,\gamma (\t), p (\t))\right\rangle
	\\
	& &   +\frac{1}{\epsilon}\left\langle Dd_\Omega(\gamma(\t)), D_{pp}^2H (\t,\gamma(\t), p(\t))Dd_\Omega(\gamma (\t))\right\rangle.
\end{eqnarray*}
We now use the growth properties of $H$  in \eqref{40}, and \eqref{41}-\eqref{39}, the lower bound for $D_{pp}^2H$ in \eqref{h1b}, and the regularity of the boundary of $\Omega$ to obtain:
\begin{align*}
\frac{1}{\mu\epsilon_k}\leq C(\mu,M')(1+|p(\t)|)^2+\kappa C(\mu,M')(1+|p(\t)|)\leq  C(\mu,M',\kappa)(1+|p(\t)|^2),
\end{align*}
where the constant $C(\mu,M',\kappa)$ depends only on $\mu$, $M'$ and $\kappa$. 
By our estimate for $p$ in \eqref{65} we get:
\begin{align*}
\frac{1}{\mu\epsilon_k} \leq C(\mu,M',\kappa)\left[1+4\mu\frac{C_1}{\delta^2}+\frac{4\mu}{\epsilon_k} d_{{\Omega}}(\gamma(\t))\right], \ \ \forall \ k\geq 1,
\end{align*}
where $C_1=8\mu+8\mu||Dg||_\infty^2+2C(\mu,M')+\kappa(T+4\mu K)$.\\
\underline{Conclusion}\\
Let $\rho=\min\left\{\rho_0,\frac{1}{32 C(\mu,M',\kappa) \mu^2}\right\}$. Owing to Lemma \ref{lemma3.1}, for all $\epsilon\in(0, \epsilon(\rho)]$ we have that 
\begin{equation*}
\sup_{t\in[0,T]} d_\Omega(\gamma(t))\leq \rho, \ \  \ \ \forall \gamma\in\mathcal{X}_{\epsilon,\delta}[x].
\end{equation*}
Hence, using \eqref{mu}, we deduce that
\begin{equation*}
\frac{1}{2\mu\epsilon_k}\leq 4 C(\mu,M',\kappa)\left[1+4\mu\frac{C_1}{\delta^2}\right].
\end{equation*}
Since the above inequality fails for $k$ large enough, we conclude that \eqref{contra} cannot hold true. So, $\gamma(t)$ belongs to $\overline{\Omega}$ for all $t\in[0,T]$.
 \end{proof}
 \noindent
 An obvious consequence of Lemma \ref{lemma5} is the following:
 \begin{corollario}\label{coro711} Fix $\delta$ as in \eqref{deltao} and take $\epsilon=\epsilon_1$, where $\epsilon_1$ is defined as in Lemma \ref{lemma5}. Then an arc $\gamma(\cdot)$ is a solution of problem (\ref{v1}) if and only if it is also a solution of (\ref{M}).
 \end{corollario}
 \noindent
 We are now ready to complete the proof of Theorem \ref{51}.
 \begin{proof}[Proof of Theorem \ref{51}]
	Let $ x\in\overline{\Omega}$ and $\gamma^\star\in \mathcal{X}[ x]$. By Corollary \ref{coro711} we have that $\gamma^\star$ is a solution of problem \eqref{v1} with $\delta$ as in \eqref{deltao} and $\epsilon=\epsilon_1$ as in Lemma \ref{lemma5}. Let $p(\cdot)$ be the associated adjoint map such that $(\gamma^\star(\cdot),p(\cdot))$ satisfies \eqref{s3}. Moreover, let $\lambda(\cdot)$ and $\beta$ be defined as in Lemma \ref{lemma.1.2}. Define $\nu= \frac{\beta}{\delta}$. Then we have $0\leq \nu \leq \frac{1}{\delta}$ and, by \eqref{s3}, 
	\begin{equation}\label{pT}
	p(T)= Dg(\gamma^\star(T))+ \nu \ D\d(\gamma^\star(T)). 
	\end{equation}
	By Lemma \ref{lemma.1.2} $\gamma^\star\in C^{1,1}([0,T];\overline{\Omega})$ and 
	\begin{equation}\label{gs}
	\dot{\gamma}^\star(t)=-D_pH(t,\gamma^\star(t),p(t)), \ \ \ \forall\ t\in[0,T].
	\end{equation}
	Moreover, $p(\cdot)\in \L(0,T;\mathbb{R}^n)$ and by \eqref{65} one has that
	\begin{equation*}
		|p(t)|\leq 2\frac{\sqrt{\mu C_1}}{\delta}, \ \ \ \forall t\in[0,T],
	\end{equation*}
	where $C_1=8\mu+8\mu||Dg||_\infty^2+2C(\mu,M')+\kappa(T+4\mu K)$. Hence, $p$ is bounded. By \eqref{gs}, and by \eqref{312} one has that
	\begin{equation*}
		||\dot{\gamma}^\star||_\infty=\sup_{t\in [0,T]}|D_p H(t,\gamma^\star(t),p(t))|\leq C(\mu,M')\Big(\sup_{t\in [0,T]} |p(t)|+1\Big)\leq C(\mu,M')\Big(2\frac{\sqrt{\mu C_1}}{\delta}+1\Big)\Big) \hspace*{-1mm}=\hspace*{-1mm}L^\star,
	\end{equation*}
where $L^\star=L^\star(\mu,M',M,\kappa,T,||Dg||_\infty,||g||_\infty)$.
	Thus, \eqref{lstar} holds\\ 	
 Finally, we want to find an explicit expression for $\lambda(t)$. 
 	For this, we set 
 	\begin{equation*}
 	D=\Big\{t \in[0,T]: \gamma^\star(t)\in\partial\Omega\Big\}\; {\rm and}\; D_{\rho_0}=\Big\{t \in[0,T]: |\d(\gamma^\star(t))|<\rho_0\Big\},
 	\end{equation*}
 	where $\rho_0$ is as in assumption \eqref{dn}.
 	Note that $\psi(t):= \d\circ\gamma^\star$ is of class $C^{1,1}$ on the open set $D_{\rho_0}$, with
 	\begin{equation*}
 	\dot \psi(t)=\Big\langle D\d(\gamma^\star(t)),\dot \gamma^\star(t)\Big\rangle=\Big\langle D\d(\gamma^\star(t)),-D_pH(t,\gamma^\star(t),p(t)) \Big\rangle.
 	\end{equation*}
 	Since $p\in \L(0,T;\mathbb{R}^n)$, $\dot{\psi}$ is absolutely continuous on $D_{\rho_0}$ with 
 	\begin{eqnarray*}
 	\ddot{\psi}(t)&=&-\Big\langle D^2\d(\gamma^\star(t))\dot \gamma^\star(t),D_pH\big(t,\gamma^\star(t),p(t)\big)\Big\rangle - \Big\langle D\d(\gamma^\star(t)),D_{pt}^2 H\big(t,\gamma^\star(t),p(t)\big)\Big\rangle\\
 	&-& \Big\langle D\d(\gamma^\star(t)),D_{px}^2 H\big(t,\gamma^\star(t),p(t)\big)\dot \gamma^\star(t)\Big\rangle-\Big\langle D\d(\gamma^\star(t)),D_{pp}^2 H\big(t,\gamma^\star(t),p(t)\big)\dot{p}(t)\Big\rangle\\
 	&=&\Big\langle D^2\d(\gamma^\star(t))D_pH\big(t,\gamma^\star(t),p(t)\big),D_pH\big(t,\gamma^\star(t),p(t)\big)\Big\rangle \\
 	&-& \Big\langle D\d(\gamma^\star(t)),D_{pt}^2 H\big(t,\gamma^\star(t),p(t)\big)\Big\rangle\\
 	&+& \Big\langle D\d(\gamma^\star(t)),D_{px}^2 H\big(t,\gamma^\star(t),p(t)\big)D_pH\big(t,\gamma^\star(t),p(t)\big)\Big\rangle\\
 &-&\Big\langle D\d(\gamma^\star(t)),D_{pp}^2H\big(t,\gamma^\star(t),p(t)\big)D_xH\big(t,\gamma^\star(t),p(t)\big)\rangle\\
 	&+&\frac{\lambda(t)}{\epsilon}\ \Big\langle D\d(\gamma^\star(t)),D_{pp}^2H\big(t,\gamma^\star(t),p(t)\big)D\d(\gamma^\star(t))\Big\rangle.
 	\end{eqnarray*}
 	Let $N_{\gamma^\star}=\{t\in D\cap (0,T)|\ \dot{\psi}(t)\neq 0\}$. Let $t\in N_{\gamma^\star}$, then there exists $\sigma>0$ such that $\gamma^\star(s)\notin \partial\Omega$ for any $s\in ((t-\sigma,t+\sigma)\setminus\{t\})\cap (0,T)$. Therefore, $N_{\gamma^\star}$ is composed of isolated points and so it is a discrete set. Hence, $\dot{\psi}(t)=0$ a.e. $t\in D\cap (0,T)$. So, $\ddot{\psi}(t)=0$ a.e.  in $D$, because $\dot \psi$ is absolutely continuous. 
 	Moreover, since $D_{pp}^2 H(t,x,p)>0$ and $|D\d(\gamma^\star(t))|=1$, we have that 
 	$$
 	\Big\langle D\d(\gamma^\star(t)),D_{pp}^2 H\big(t,\gamma^\star(t),p(t)\big)D\d(\gamma^\star(t))\Big\rangle>0, \qquad {\rm a.e.} \;t\in D_{\rho_0}.
 	$$ 
 	So, for a.e. $t\in D$, $\lambda(t)$ is given by
 	\begin{align*}
 	\frac{\lambda(t)}{\epsilon} =&\frac{1}{\langle D\d(\gamma^\star(t)),D_{pp}^2 H(t,\gamma^\star(t),p(t))D\d(\gamma^\star(t))\rangle}\ \Big[ \Big\langle D\d(\gamma^\star(t)),D_{pt}^2 H\big(t,\gamma^\star(t),p(t)\big)\Big\rangle\\
 	&-\Big\langle D^2\d(\gamma^\star(t))D_pH\big(t,\gamma^\star(t),p(t)\big),D_pH\big(t,\gamma^\star(t),p(t)\big)\Big\rangle\\
 	&- \Big\langle D\d(\gamma^\star(t)),D_{px}^2 H\big(t,\gamma^\star(t),p(t)\big)D_pH\big(t,\gamma^\star(t),p(t)\big)\Big\rangle\\
 	&+\Big\langle D\d(\gamma^\star(t)),D_{pp}^2H\big(t,\gamma^\star(t),p(t)\big)D_xH\big(t,\gamma^\star(t),p(t)\big) \Big\rangle
 	\Big].
 	\end{align*}
 	Since $\lambda(t)=0$ for all $t\in[0,T]\setminus D$ by \eqref{lambda}, taking $\Lambda(t,\gamma^\star(t),p(t))=\frac{\lambda(t)}{\epsilon}$, we obtain the conclusion.
 	 \end{proof}
 	 \begin{remark}\label{lambda1}
 	The above proof gives a representation of $\Lambda$, i.e., for all $(t,x,p) \in[0,T]\times\Sigma_{\rho_0}\times \mathbb{R}^n$ one has that 
 	   	\begin{align*}
 	   	\Lambda(t,x,p)=
 	   	&\frac{1}{\theta(t,x,p)}\ \Big[
 	   	-\Big\langle D^2\d(x)D_pH\big(t,x,p\big),D_pH\big(t,x,p\big)\Big\rangle- \Big\langle D\d(x),D_{pt}^2 H\big(t,x,p\big)\Big\rangle- \\
 	   	&\Big\langle D\d(x),D_{px}^2 H\big(t,x,p\big)D_pH\big(t,x,p\big)\Big\rangle+\Big\langle D\d(x),D_{pp}^2H\big(t,x,p\big)D_xH\big(t,x,p\big) \Big\rangle
 	   	\Big],
 	   	\end{align*}
 	   	where $\theta(t,x,p):=\langle D\d(x),D_{pp}^2 H(t,x,p)D\d(x)\rangle$. Observe that \eqref{h1b} ensures that $\theta(t,x,p)>0$ for all $t\in [0,T]$, for all $x\in\Sigma_{\rho_0}$ and for all $p\in \mathbb{R}^n$.
 	 \end{remark}
 	\subsection{Proof of Theorem \ref{51} for general $U$}
 	We now want to remove the extra assumption $U=\mathbb{R}^n$. For this purpose, it suffices to show that the data $f$ and $g$---a priori defined just on $U$---can be extended to $\mathbb{R}^n$ preserving the conditions in (f0)-(f2) and (g1). So, we proceed to construct such an extension by taking a cut-off function $\xi\in C^\infty(\mathbb{R})$ such that
 	\begin{align}\label{xi}
 	\begin{cases}
 	\xi(x)=0 \ \ \ \ &\mbox{if} \ \ x\in (-\infty,\frac{1}{3}],\\
 	0<\xi(x)<1 &\mbox{if}\ \ x\in (\frac{1}{3},\frac{2}{3}),\\
 	\xi=1 &\mbox{if} \ \ x \in [\frac{2}{3},+\infty).
 	\end{cases}
 	\end{align}
 	\begin{lemma}
 		Let $\Omega\subset\mathbb{R}^n$ be a bounded open set with $C^2$ boundary. Let $U$ be a open subset of $\mathbb{R}^n$ such that $\overline{\Omega}\subset U$ and 
 		\begin{equation*}
 		0<dist(\overline{\Omega}, \mathbb{R}^n\setminus U)=:\sigma_0.
 		\end{equation*}
 		Suppose that $f:[0,T]\times U\times \mathbb{R}^n\rightarrow \mathbb{R}$ and $g:U\rightarrow \mathbb{R}$ satisfy (f0)-(f2) and (g1), respectively.
 		Set $\sigma=\sigma_0\wedge \rho_0$. Then, the function $f$ admits the extension
 		\begin{equation*}
 		\widetilde{f}(t,x,v)= \xi\left(\frac{\d(x)}{\sigma}\right)\frac{|v|^2}{2}+ \left (1-\xi\left(\frac{\d(x)}{\sigma}\right)\right)f(t,x,v), \ \ \ \forall \ (t,x,v)\in[0,T]\times\mathbb{R}^n\times\mathbb{R}^n,
 		\end{equation*}
 		that verifies the conditions (f0)-(f2) with $U=\mathbb{R}^n$. Moreover, the function $g$ admits the extension
 $$
 \widetilde{g}(x)= \left( 1-\xi\left(\frac{\d(x)}{\sigma}\right)\right)g(x),\ \  \ \ \forall x\in\mathbb{R}^n,
 $$
 that satisfies the condition (g1) with $U=\mathbb{R}^n$.		
 	\end{lemma}
 	\begin{proof}
 By construction we note that $\widetilde{f}\in C([0,T]\times \mathbb{R}^n\times\mathbb{R}^n)$. Moreover, for all $t\in[0,T]$ the function  $(x,v)\longmapsto \widetilde{f}(t,x,v)$ is differentiable and the map $(x,v)\longmapsto D_{v}\widetilde{f}(t,x,v)$ is continuously differentiable by construction. Furthermore, $D_x\widetilde{f}$, $D_v\widetilde{f}$ are continuous on $[0,T]\times \mathbb{R}^n\times \mathbb{R}^n$ and $\widetilde{f}$ satisfies \eqref{bm}.
In order to prove \eqref{f2} for $\widetilde{f}$, we observe that
 \begin{align*}
 D_v\widetilde{f}(t,x,v)= \xi\left(\frac{\d(x)}{\sigma}\right) v + \left (1-\xi\left(\frac{\d(x)}{\sigma}\right)\right)D_vf(t,x,v),
 \end{align*}
 and
\begin{align*}
D_{vv}^2\widetilde{f}(t,x,v)=\xi\left(\frac{\d(x)}{\sigma}\right)I+\left (1-\xi\left(\frac{\d(x)}{\sigma}\right)\right)D_{vv}^2f(t,x,v).
\end{align*}
Hence, by the definition of $\xi$ and \eqref{f2} we obtain that
\begin{equation*}
\Big(1\wedge\frac{1}{\mu}\Big)I\leq D_{vv}^2\widetilde{f}(t,x,v)\leq (1\vee \mu) I,\ \  \ \ \ \forall \ (t,x,v)\in[0,T]\times\mathbb{R}^n\times \mathbb{R}^n.
\end{equation*}
Since $\mu\geq 1$, we have that $\widetilde{f}$ verifies the estimate in \eqref{f2}.\\
Moreover, since 
\begin{eqnarray*}
D_x(D_v\widetilde{f}(t,x,v))&=&\dot{\xi}\left(\frac{\d(x)}{\sigma}\right)v\otimes\frac{D\d(x)}{\sigma}+\left(1-\xi\left(\frac{\d(x)}{\sigma}\right)\right)D_{vx}^2f(t,x,v)\\
&-&\dot{\xi}\left(\frac{\d(x)}{\sigma}\right)D_vf(t,x,v)\otimes\frac{D\d(x)}{\sigma},
\end{eqnarray*}
and by \eqref{fvx} we obtain that
\begin{equation*}
||D_{vx}^2\widetilde{f}(t,x,v)||\leq C(\mu,M)(1+|v|) \ \ \ \forall (t,x,v)\in [0,T]\times \mathbb{R}^n\times\mathbb{R}^n.
\end{equation*}
For all $(x,v)\in\mathbb{R}^n\times\mathbb{R}^n$ the function $t\longmapsto\widetilde{f}(t,x,v)$ and the map $t\longmapsto D_v\widetilde{f}(t,x,v)$ are Lipschitz continuous by construction. Moreover, by \eqref{lf1} and the definition of $\xi$ one has that
  \begin{equation*}
  \Big|\widetilde{f}(t,x,v)-\widetilde{f}(s,x,v)\Big|=\left| \left (1-\xi\left(\frac{\d(x)}{\sigma}\right)\right) \big[ f(t,x,v)-f(s,x,v)\big]\right|\leq \kappa(1+|v|^2)|t-s|
  \end{equation*}
for all $t$, $s\in [0,T]$, $x\in\mathbb{R}^n$, $v\in\mathbb{R}^n$. 
 Now, we have to prove that \eqref{fvt} holds for $\widetilde{f}$. Indeed, using \eqref{fvt} we deduce that
 		\begin{align*}
 		&\big|D_v\widetilde{f}(t,x,v))-D_v\widetilde{f}(s,x,v))\big|\leq \left|\left (1-\xi\left(\frac{\d(x)}{\sigma}\right)\right)\big[D_{v}f(t,x,v)-D_vf(s,x,v))\big]\right|\\
 		&\leq \kappa(1+|v|)|t-s|,
 		\end{align*}
 for all $t$, $s\in[0,T]$, $x\in\mathbb{R}^n$, $v\in \mathbb{R}^n$. Therefore, $\widetilde{f}$ verifies the assumptions (f0)-(f2).\\
 Finally, by the regularity of $\d$, $\xi$, and $g$ we have that $\widetilde{g}$ is of class $C^1_b(\mathbb{R}^n)$. This completes the proof.
 	\end{proof}
\section{Applications of Theorem \ref{51}}
\subsection{Lipschitz regularity for constrained minimization problems}
Suppose that $f:[0,T]\times U\times \mathbb{R}^n\rightarrow \mathbb{R}$ and $g:U\rightarrow \mathbb{R}$ satisfy the assumptions (f0)-(f2) and (g1), respectively. Let $(t,x)\in [0,T]\times \overline{\Omega}$.
Define $u:[0,T]\times \overline{\Omega}\rightarrow\mathbb{R}$ as the value function of the minimization problem \eqref{M}, i.e.,
\begin{equation}\label{vf}
u(t,x)=\inf_{\begin{array}{c}
	\gamma\in \Gamma\\
	\gamma(t)=x
	\end{array}} \int_{t}^T f(s,\gamma(s),\dot{\gamma}(s)) \,ds + g(\gamma(T)). 
\end{equation}
\begin{proposition}\label{fw}
	Let $\Omega$ be a bounded open subset of $\mathbb{R}^n$ with $C^2$ boundary. Suppose that $f$ and $g$ satisfy (f0)-(f2) and (g1), respectively. Then, $u$ is Lipschitz continuous in $(0,T)\times\overline{\Omega}$.
\end{proposition}
\begin{proof}
	First, we shall prove that $u$ is Lipschitz continuous in space, uniformly in time.
	Let $x_0\in\Omega$ and choose $0<r<1$ such that $B_r(x_0)\subset B_{2r}(x_0)\subset B_{4r}(x_0)\subset \Omega$. To prove that $u$ is Lipschitz continuous in $B_r(x_0)$, we take $x\neq y$ in $B_r(x_0)$ and t $\in (0,T)$. Let $\gamma$ be an optimal trajectory for $u$ at $(t,x)$ and let $\oh$ be the trajectory defined by
	\begin{align*}
	\begin{cases}
	\oh(t)=y,\\
	\dot{\oh}(s)=\dot{\gamma}(s) +\frac{x-y}{\tau} \ \ &\mbox{if} \ s\in[t,t+\tau] \ \ \text{a.e.},\\
	\dot{\oh}(s)=\dot{\gamma}(s) \ \ &\mbox{otherwise},
	\end{cases}
	\end{align*}
	where $\tau=\frac{|x-y|}{2L^\star}<T-t$.
	We claim that
	\begin{enumerate}
		\item[(a)] $\oh(t+\tau)=\gamma(t+\tau)$;
		\item[(b)] $\oh(s)=\gamma(s)$ for any $s\in[t+\tau,T]$;
		\item[(c)] $|\oh(s)-\gamma(s)|\leq |y-x|$ for any $s \in [t,t+\tau]$;
		\item[(d)] $\oh(s)\in \overline{\Omega}$ for any $s \in [t,T]$.
	\end{enumerate}
	Indeed, by the definition of $\oh$ we have that
	\begin{align*}
	\oh(t+\tau)-\oh(t)=\oh(t+\tau)-y=\int_t^{t+\tau}\Big(\dot{\gamma}(s)+\frac{x-y}{\tau}\Big)\,ds=\gamma(t+\tau)-y,
	\end{align*}
	and this gives (a).\\
	Moreover, by (a), and by the definition of $\oh$ one has that $\oh(s)=\gamma(s)$ for any $s\in[t+\tau,T]$. Hence, $\oh$ verifies (b).\\
	By the definition of $\oh$, for any $s\in [t,t+\tau]$ we obtain that
	\begin{align*}
	\Big|\oh(s)-\gamma(s)\Big|\leq\Big|y-x+\int_t^s (\dot\oh(\sigma)- \dot{\gamma}(\sigma)) \,d\sigma\Big|=\Big|y-x+ \int_t^{s}\frac{x-y}{\tau} \,d\sigma\Big|\leq |y-x|
	\end{align*}
	and so (c) holds.\\
	Since $\gamma$ is an optimal trajectory for $u$ and by $\oh(s)=\gamma(s)$ for all $s\in [t+\tau,T]$, we only have to prove that $\oh(s)$ belongs to $\overline{\Omega}$ for all $s\in[t,t+\tau]$. Let $s\in [t,t+\tau]$, by Theorem \ref{51} one has that
	\begin{align*}
	&|\oh(s)-x_0|\leq|\oh(s)-y|+|y-x_0|\leq\left|\int_t^s \dot{\oh}(\sigma)\,d\sigma\right|+r\leq \int_t^s \ \Big|\dot{\gamma}(\sigma)+\frac{x-y}{\tau}\Big|\,d\sigma+r\\
	&\leq \int_t^s \Big[ |\dot{\gamma}(\sigma)|+ \frac{|x-y|}{\tau}\Big]\,d\sigma+r 
	\leq L^\star (s-t)+\frac{|x-y|}{\tau}(s-t)+r\leq L^\star\tau +|x-y|+r.
	\end{align*}
	Recalling that $\tau=\frac{|x-y|}{2L^\star}$ one has that
	\begin{equation*}
	|\oh(s)-x_0|\leq \frac{|x-y|}{2}+|x-y|+r\leq 4r.
	\end{equation*}
	Therefore, $\oh(s)\in B_{4r}(x_0)\subset \overline{\Omega}$ for all $s\in[t,t+\tau]$.\\
	Using the dynamic programming principle, by (a) one has that
	\begin{equation}
	u(t,y)\leq \int_t^{t+\tau} f(s,\oh(s),\dot{\oh}(s))\,ds + u(t+\tau,\gamma(t+\tau)).
	\end{equation}
	Since $\gamma$ is an optimal trajectory for $u$ at $(t,x)$, we obtain that
	\begin{equation*}
	u(t,y)\leq u(t,x) +\int_t^{t+\tau} \Big[f(s,\oh(s),\dot\oh(s))-f(s,\gamma(s),\dot{\gamma}(s))\Big] \,ds.
	\end{equation*}
	By \eqref{l4}, \eqref{lx}, and the definition of $\oh$, for $s\in [t,t+\tau]$ we have that
	\begin{align*}
	&|f(s,\oh(s),\dot{\oh}(s))-f(s,\gamma(s),\dot{\gamma}(s))|\\
	&\leq|f(s,\oh(s),\dot{\oh}(s))-f(s,\oh(s),\dot{\gamma}(s))|+|f(s,\oh(s),\dot{\gamma}(s))-f(s,\gamma(s),\dot{\gamma}(s))|\\	
	&\leq \int_0^1 |\langle D_vf(s,\oh(s),\lambda\dot{\oh}(s)+(1-\lambda)\dot{\gamma}(s)),\dot{\oh}(s)-\dot{\gamma}(s)\rangle|\,d\lambda\\
	& + \int_0^1|D_xf(s,\lambda\oh(s)+(1-\lambda)\gamma(s),\dot{\gamma}(s)),\oh(s)-\gamma(s)\rangle|\,d\lambda\\
	&\leq C(\mu,M)|\dot{\oh}(s)-\dot{\gamma}(s)|\int_0^1 (1+|\lambda\dot{\oh}(s)+(1-\lambda)\dot{\gamma}(s)|)\,d\lambda \\
	&+ C(\mu,M)|\oh(s)-\gamma(s)|\int_0^1(1+ |\dot{\gamma}(s)|^2)\,d\lambda.
	\end{align*}
	By Theorem \ref{51} one has that
	\begin{align}
	&\int_0^1 (1+|\lambda\dot{\oh}(s)+(1-\lambda)\dot{\gamma}(s)|)\,d\lambda\leq 1+4L^\star,\label{1e}\\
	&\int_0^1(1+ |\dot{\gamma}(s)|^2)\,d\lambda\leq 1+(L^\star)^2.\label{21e}
	\end{align}	
	Using \eqref{1e}, \eqref{21e}, and (c), by the definition of $\overline{\gamma}$ one has that
	\begin{equation}\label{bl1}
	|f(s,\oh(s),\dot{\oh}(s))-f(s,\gamma(s),\dot{\gamma}(s))|\leq C(\mu,M)(1+4L^\star)\frac{|x-y|}{\tau}+C(\mu,M)(1+(L^\star)^2)|x-y|,
	\end{equation}
	for a.e. $s\in[t,t+\tau]$.\\
	By \eqref{bl1}, and the choice of $\tau$ we deduce that
	\begin{align*}
	&u(t,y)\leq u(t,x) + C(\mu,M)(1+4L^\star)\int_t^{t+\tau}\frac{|x-y|}{\tau}  \,ds+ C(\mu,M)(1+(L^\star)^2)\int_t^{t+\tau} |x-y|\,ds\\
	&\leq u(t,x) + C(\mu,M)(1+4L^\star)\big|x-y\big|+\tau C(\mu,M)(1+(L^\star)^2)\big|x-y\big|\leq u(t,x)+C_{L^\star}|x-y|
	\end{align*}		
	where $C_{L^\star}=C(\mu,M)(1+4L^\star)+\frac{1}{2L^\star}C(\mu,M)(1+(L^\star)^2)$. Thus, $u$ is locally Lipschitz continuous in space and one has that $||Du||_\infty\leq \vartheta$, where $\vartheta$ is a constant not depending on $\Omega$. Owing to the smoothness of $\Omega$, $u$ is globally Lipschitz continuous in space.\\
	Let $x \in \overline\Omega$. Let $t_1$, $t_2 \in (0,T)$ and,
	without loss of generality, suppose that $t_2\geq t_1$. Let $\gamma$ be an optimal trajectory for $u$ at $(t_1,x)$.
	Then, 
	\begin{equation}\label{3e}
	|u(t_2,x)-u(t_1,x)|\leq |u(t_2,x)-u(t_2,\gamma(t_2))|+|u(t_2,\gamma(t_2))-u(t_1,x)|.
	\end{equation}
	The first term on the right-side of \eqref{3e} can be estimated using the Lipschitz continuity in space of $u$ and Theorem \ref{51}. Indeed, we get
	\begin{equation}\label{4}
	|u(t_2,x)-u(t_2,\gamma(t_2))|\leq C_{L^\star}|x-\gamma(t_2)| \leq  C_{L^\star}\int_{t_1}^{t_2}|\dot{\gamma}(s)|\,ds\leq L^\star C_{L^\star} (t_2-t_1).
	\end{equation}
	We only have to estimate the second term on the right-side of \eqref{3e}.
	By dynamic programming principle, \eqref{l3}, and the assumptions on $F$ we deduce that
	\begin{align}\label{5}
	|u(t_2,\gamma(t_2))-u(t_1,x)|&\leq \Big |\int_{t_1}^{t_2}f(s,\gamma(s),\dot{\gamma}(s))\,ds\Big|\leq \int_{t_1}^{t_2}|f(s,\gamma(s),\dot{\gamma}(s))|\,ds\\
	&\leq \int_{t_1}^{t_2} \Big[C(\mu,M)+ 4\mu |\dot{\gamma}(s)|^2\Big]\,ds\leq \Big[C(\mu,M)+4\mu L^\star\Big] (t_2-t_1)\nonumber
	\end{align}
	Using \eqref{4} and \eqref{5} in \eqref{3e}, we obtain that $u$ is Lipschitz continuous in time. This completes the proof.
\end{proof}
 \subsection{Lipschitz regularity for constrained MFG equilibria}
 	 In this section we want to apply Theorem \ref{51} to a mean field game (MFG) problem with state constraints. Such a problem was studied in \cite{cc}, where the existence and uniqueness of constrained equilibria was obtained under fairly general assumptions on the data. Here, we will apply our necessary conditions to deduce the existence of more regular equilibria than those constructed in \cite{cc}, assuming the data $F$ and $G$ to be Lipschitz continuous.\\
 	 \underline{\textit{Assumptions}}\\
 	 Let $\Omega$ be a bounded open subset of $\mathbb{R}^n$ with $C^2$ boundary. Let $\mathcal{P}(\overline{\Omega})$  be the set of all Borel probability measures on $\overline\Omega$ endowed with the Kantorovich-Rubinstein distance $d_1$ defined in \eqref{2.7}. Let $U$ be an open subset of $\mathbb{R}^n$ and such that $\overline{\Omega}\subset U$. Assume that $F:U\times\mathcal{P}(\overline{\Omega})\rightarrow \mathbb{R}$ and $G:U\times \mathcal{P}(\overline{\Omega})\rightarrow \mathbb{R}$ satisfy the following hypotheses.
\begin{enumerate}
\item[(D1)] For all $x\in U$, the functions $m\longmapsto F(x,m)$ and  $m\longmapsto G(x,m)$ are Lipschitz continuous, i.e., there exists a constant $\kappa\geq 0$ such that
 	 	\begin{align}
 	 	|F(x,m_1)-F(x,m_2)|+ |G(x,m_1)-G(x,m_2)| \leq \kappa d_1(m_1,m_2),\label{lf} 
 	  	 	\end{align}
 	 	for any $m_1$, $m_2 \in\mathcal{P}(\overline{\Omega})$.
\item[(D2)] For all $m\in \mathcal{P}(\overline{\Omega})$, the functions $x\longmapsto G(x,m)$ and $x\longmapsto F(x,m)$ belong to $C^1_b(U)$. Moreover
\begin{equation*}
|D_xF(x,m)|+|D_xG(x,m)|\leq \kappa, \ \ \ \ \forall \ x\in U, \ \forall \ m\in \mathcal{P}(\overline{\Omega}).
\end{equation*}
\end{enumerate}
\noindent
 	 Let $L:U\times\mathbb{R}^n\rightarrow \mathbb{R}$ be a function that satisfies the following assumptions.
 	 \begin{enumerate}
 	 \item[(L0)] $L\in C^1(U\times \mathbb{R}^n)$ and there exists a constant $M\geq 0$ such that
 	 \begin{equation}\label{bml}
 	 |L(x,0)|+|D_xL(x,0)|+|D_vL(x,0)|\leq M, \ \ \ \ \forall \ x\in U.
 	 \end{equation}
 	 \item[(L1)] $D_vL$ is differentiable on $U\times\mathbb{R}^n$ and there exists a constant $\mu\geq 1$ such that
 	 \begin{align}
 	& \frac{I}{\mu} \leq D^2_{vv}L(x,v)\leq I\mu,\label{lh1}\\
 	& ||D_{vx}^2L(x,v)||\leq \mu(1+|v|),\label{c6}
 	 \end{align}
for all $(x,v)\in U\times \mathbb{R}^n$.
\end{enumerate}
\begin{remark}
(i) $F$, $G$ and $L$ are assumed to be defined on $U\times \mathcal{P}(\overline{\Omega})$ and on $U\times \mathbb{R}^n$, respectively, just for simplicity. All the results of this section hold true if we replace $U$ by $\overline{\Omega}$. This fact can be easily checked by using well-known extension techniques (see, e.g. \cite[Theorem 4.26]{adams}).\\
(ii) Arguing as Lemma \ref{r3} we deduce that there exists a positive constant $C(\mu,M)$ that dependes only on $M$, $\mu$ such that
\begin{align}
&|D_xL(x,v)|\leq C(\mu,M)(1+|v|^2),\label{c4}\\
&|D_vL(x,v)|\leq C(\mu,M)(1+|v|),\label{l41},\\
&\frac{|v|^2}{4\mu}-C(\mu,M) \leq L(x,v)\leq 4\mu|v|^2 +C(\mu,M),\label{l31}
\end{align}
for all $(x,v)\in U\times\mathbb{R}^n$.
\end{remark}
%
 	\noindent
Let $m\in \L(0,T;\mathcal{P}(\overline{\Omega}))$. If we set $f(t,x,v)=L(x,v)+F(x,m(t))$, then the associated Hamiltonian $H$ takes the form 
 	 	$$
 	 	H(t,x,p)=\hl(x,p)-F(x,m(t)),  \ \  \ \forall\ (t,x,p)\in[0,T]\times U\times\mathbb{R}^n,
 	 	$$ 
 	 	where
 	 \begin{equation*}
 	 \hl(x,p)=\sup_{v\in\mathbb{R}^n}\Big\{-\langle p,v\rangle-L(x,v)\Big\},\ \  \ \ \ \forall\ (x,p)\in U\times\mathbb{R}^n.
 	 \end{equation*}
The assumptions on $L$ imply that $\hl$ satisfies the following conditions.
 	 	\begin{enumerate}
 	 \item[1.] 
 	 	$H_L\in C^1(U\times \mathbb{R}^n)$ and there exists a constant $M'\geq 0$ such that
 \begin{equation}
 |\hl(x,0)|+|D_x\hl(x,0)|+|D_p\hl(x,0)|\leq M', \ \ \ \ \forall x\in U.
 \end{equation}
\item[2.] $D_p\hl$ is differentiable on $U\times\mathbb{R}^n$ and satisfies
 	 	\begin{align}
 	 	 &	\frac{I}{\mu}\leq D_{pp}\hl(x,p)\leq I\mu, \ \ \ \ \ \ \ \ \ \ \ \ \ \ \forall \ (x,p)\in U\times\mathbb{R}^n,\label{h1}\\
 	   &||D_{px}^2  \hl(x,p)||\leq C(\mu,M')(1+|p|), \ \ \ \forall \ (x,p)\in U\times \mathbb{R}^n,
 	 	 	\end{align}
where $\mu$ is the constant in (L1) and $C(\mu,M')$ depends only on $\mu$ and $M'$. 
 	 	\end{enumerate}
 	 \noindent  
 	 For any $t\in [0,T]$, we denote by $e_t:\Gamma\to \overline \Omega$ the evaluation map defined by 
 	 \begin{equation*}
 	 e_t(\gamma)= \gamma(t), \ \ \ \ \forall \gamma\in\Gamma.
 	 \end{equation*}
 	 For any $\eta\in\mathcal{P}(\Gamma)$, we define 
 	 \begin{equation*}
 	 m^\eta(t)=e_t\sharp\eta \ \ \ \ \forall t\in [0,T].
 	 \end{equation*}
 	  \begin{remark}\label{r45}
 	  	We observe that for any $\eta\in\mathcal{P}(\Gamma)$, the following holds true (see \cite{cc} for a proof).
 	  	\begin{enumerate}
 	  		\item[(i)] $m^\eta\in C([0,T];\mathcal{P}(\overline{\Omega}))$.
 	  		\item[(ii)] Let $\eta_i$, $\eta\in\mathcal{P}(\Gamma)$, $i\geq 1$, be such that $\eta_i$ is narrowly convergent to $\eta$. Then $m^{\eta_i}(t)$ is narrowly convergent to $m^\eta(t)$ for all $t\in[0,T]$.
 	  	\end{enumerate}
 	  \end{remark}
 	 \noindent
 	 For any fixed $m_0\in\mathcal{P}(\overline{\Omega})$, we denote by ${\mathcal P}_{m_0}(\Gamma)$ the set of all Borel probability measures $\eta$ on $\Gamma$ such that $e_0\sharp \eta=m_0$.
 	 For all $\eta \in \mathcal{P}_{m_0}(\Gamma)$, we set
 	 \begin{equation*}
 	 J_\eta [\gamma]=\int_0^T \Big[L(\gamma(t),\dot \gamma(t))+ F(\gamma(t),m^\eta(t))\Big]\ dt + G(\gamma(T),m^\eta(T)), \ \ \ \ \  \forall \gamma\in\Gamma.
 	 \end{equation*}
For all $x \in \overline{\Omega}$ and $\eta\in\mathcal{P}_{m_0}(\Gamma)$, we define
 	   \begin{equation*}
 	   \Gamma^\eta[x]=\Big\{ \gamma\in\Gamma[x]:J_\eta[\gamma]=\min_{\Gamma[x]} J_\eta\Big\}.
 	   \end{equation*}
It is shown in \cite{cc} that, for every $\eta\in \mathcal{P}_{m_0}(\Gamma)$, the set $\Gamma^\eta[x]$ is nonempty and $\Gamma^\eta[\cdot]$ has closed graph.\\
 	   We recall the definition of constrained MFG equilibria for $m_0$, given in \cite{cc}.
 	   \begin{definition}
 	   	Let $m_0\in\mathcal{P}(\overline{\Omega})$. We say that $\eta\in\mathcal{P}_{m_0}(\Gamma)$ is a contrained MFG equilibrium for $m_0$ if
 	   	\begin{equation*}
 	   	supp(\eta)\subseteq \bigcup_{x\in\overline{\Omega}} \Gamma^\eta[x].
 	   	\end{equation*}
 	   \end{definition}
 	   \noindent
Let $\Gamma'$ be a nonempty subset of $\Gamma$. We denote by $\mathcal{P}_{m_0}(\Gamma')$ the set of all Borel probability measures $\eta$ on $\Gamma'$ such that $e_0\sharp\eta=m_0$. We now introduce special subfamilies of $\mathcal{P}_{m_0}(\Gamma)$ that play a key role in what follows.
 	  \begin{definition}	
 	  Let $\Gamma'$ be a nonempty subset of $\Gamma$. We define by $\mathcal{P}_{m_0}^{\L}(\Gamma')$ the set of $\eta\in\mathcal{P}_{m_0}(\Gamma')$ such that $m^\eta(t)=e_t\sharp \eta$ is Lipschitz continuous, i.e., 
 	  \begin{equation*}
 	  \mathcal{P}_{m_0}^{\L}(\Gamma')=\{\eta\in\mathcal{P}_{m_0}(\Gamma'): m^\eta \in \L(0,T;\mathcal{P}(\overline{\Omega}))\}.
 	  \end{equation*}
 	  \end{definition}
 	  \begin{remark}\label{nonempty}
 	  We note that $\mathcal{P}^{\L}_{m_0}(\Gamma)$ is a nonempty convex set. Indeed, let $j:\overline{\Omega}\rightarrow \Gamma$ be the continuous map defined by
 	  	$$
 	  	j(x)(t)=x \ \ \ \ \forall t \in[0,T].
 	  	$$
 	  	Then,
 	  	$$
 	  	\eta :=j\sharp m_0
 	  	$$
 	  	is a Borel probability measure on $\Gamma$ and $\eta \in\mathcal{P}^{\L}_{m_0}(\Gamma)$.\\
 	  	In order to show that $\lo$ is convex, let $\{\eta_i\}_{i=1,2}\subset \mathcal{P}_{m_0}^{\L}(\Gamma)$ and let $\lambda_1$, $\lambda_2\geq 0$ be such that $\lambda_1+\lambda_2=1$. Since $\eta_i$ are Borel probability measures, $\eta:=\lambda\eta_1+(1-\lambda)\eta_2$ is a Borel probability measure as well.  Moreover, for any Borel set $B\in \mathscr{B}(\overline{\Omega})$ we have that
 	  	\begin{equation*}
 	  	e_0 \sharp \eta (B)=\eta (e_0^{-1}(B))= \sum_{i=1}^{2} \lambda_i \eta_i(e_0^{-1}(B))=\sum_{i=1}^{2} \lambda_i e_0 \sharp \eta_i(B)=\sum_{i=1}^{2} \lambda_i m_0(B) =m_0 (B).
 	  	\end{equation*}
 	  	So, $\eta\in\mathcal{P}_{m_0}(\Gamma)$. Since $m^{\eta_1}$, $m^{\eta_2}\in \L(0,T;\mathcal{P}(\overline{\Omega}))$, we have that $m^\eta(t)=\lambda_1m^{\eta_1}(t)+\lambda_2m^{\eta_2}(t)$ belongs to $\L(0,T;\mathcal{P}(\overline{\Omega}))$. 
\end{remark}
\noindent
In the next result, we apply Theorem \ref{51} to prove a useful property of minimizers of $J_\eta$.
\begin{proposition}\label{pr1}
Let $\Omega$ be a bounded open subset of $\mathbb{R}^n$ with $C^2$ boundary and let $m_0\in\mathcal{P}(\overline{\Omega})$. Suppose that (L0), (L1), (D1), and (D2) hold true. Let $\eta\in\mathcal{P}^{\L}_{m_0}(\Gamma)$ and fix $x\in\overline{\Omega}$. Then $\Gamma^\eta[x]\subset C^{1,1}([0,T];\mathbb{R}^n)$ and 
\begin{equation}\label{l0}
||\dot{\gamma}||_\infty\leq L_0, \ \ \ \forall \gamma\in \Gamma^\eta[x],
\end{equation}
where $L_0=L_0(\mu,M',M,\kappa,T, ||G||_\infty,||DG||_\infty)$.
\end{proposition}
\begin{proof}
	Let $\eta\in\mathcal{P}_{m_0}^{\L}(\Gamma)$, $x\in\overline{\Omega}$ and $\gamma\in \Gamma^\eta[x]$. Since $m\in \L(0,T;\mathcal{P}(\overline{\Omega}))$, taking $f(t,x,v)=L(x,v)+F(x,m(t))$, one can easly check that all the assumptions of Theorem \ref{51} are satisfied by $f$ and $G$. Therefore, we have that $\Gamma^\eta[x]\subset C^{1,1}([0,T];\mathbb{R}^n)$ and, in this case, \eqref{lstar} becomes
	\begin{equation*}
	||\dot{\gamma}||_\infty\leq L_0, \ \ \ \forall \gamma\in \Gamma^\eta[x],
	\end{equation*}
	where $L_0=L_0(\mu,M',M,\kappa,T, ||G||_\infty,||DG||_\infty)$.
	\end{proof}	 
	 \noindent
	 We denote by $\Gamma_{L_0}$ the set of $\gamma\in\Gamma$ such that \eqref{l0} holds, i.e.,
	 \begin{equation}\label{tgamma}
	 \Gamma_{L_0}=\{\gamma \in \Gamma:||\dot\gamma||_\infty\leq L_0\}.
	 \end{equation}
\begin{lemma}\label{lemmaut}
Let $m_0\in \mathcal{P}(\overline{\Omega})$. Then, $\mathcal{P}_{m_0}^{\L}(\Gamma_{L_0})$ is a nonempty convex compact subset of $\mathcal{P}_{m_0}(\Gamma)$. Moreover, for every $\eta\in\mathcal{P}_{m_0}(\Gamma_{L_0})$, $m^\eta(t):=e_t\sharp \eta$ is Lipschitz continuous of constant $L_0$, where $L_0$ is as in Proposition \ref{pr1}.
\end{lemma}
\begin{proof}
Arguing as in Remark \ref{nonempty}, we obtain that $\mathcal{P}_{m_0}^{\L}(\Gamma_{L_0})$ is a nonempty convex set.
Moreover, since $\Gamma_{L_0}$ is compactly embedded in $\Gamma$, one has that $\mathcal{P}_{m_0}^{\L}(\Gamma_{L_0})$ is compact.\\
Let $\eta\in\mathcal{P}_{m_0}(\Gamma_{L_0})$ and $m^\eta(t)=e_t\sharp\eta$. For any $t_1,t_2\in[0,T]$, we recall that
\begin{equation*}
d_1(m^\eta(t_2),m^\eta(t_1))=\sup\Big\{\int_{\overline{\Omega}} \phi(x)(m^\eta(t_2, \,dx)-m^\eta(t_1,\,dx))\ \Big|\ \phi:\overline{\Omega}\rightarrow\mathbb{R}\ \ \mbox{is 1-Lipschitz} \Big\}. 
\end{equation*}
Since $\phi$ is 1-Lipschitz continuous, one has that
\begin{align*}
&\int_{\overline\Omega} \phi(x)\,(m^\eta(t_2,dx)-m^\eta(t_1,dx))=\int_{\Gamma}\Big[ \phi(e_{t_2}(\gamma))-\phi(e_{t_1}(\gamma))\Big] \,d\eta(\gamma)\\
&=\int_{\Gamma} \Big[\phi(\gamma(t_2))-\phi(\gamma(t_1))\Big] \,d\eta(\gamma) \leq \int_{\Gamma} |\gamma(t_2)-\gamma(t_1)|\,d\eta(\gamma).
\end{align*}
Since $\eta \in \mathcal{P}_{m_0}(\Gamma_{L_0})$, we deduce that
\begin{align*}
\int_{\Gamma} |\gamma(t_2)-\gamma(t_1)|\,d\eta(\gamma)\leq L_0\int_{\Gamma} |t_2-t_1|\,d\eta(\gamma)=L_0|t_2-t_1|
\end{align*}
and so $m^\eta(t)$ is Lipschitz continuous of constant $L_0$.
\end{proof}
\noindent
In the next result, we deduce the existence of more regular equilibria than those constructed in \cite{cc}.

 	 \begin{theorem}\label{teoresist}
 	 	Let $\Omega$ be a bounded open subset of $\mathbb{R}^n$ with $C^2$ boundary and $m_0\in\mathcal{P}(\overline{\Omega})$. Suppose that (L0), (L1), (D1), and (D2) hold true. Then, there exists at least one constrained MFG equilibrium $\eta \in\lo$.
 	 \end{theorem}
 	 \begin{proof}
 	 	First of all, we recall that for any $\eta\in\lo$, there exists a unique Borel measurable family \footnote{We say that $\{\eta_x\}_{x\in \overline{\Omega}}$ is a Borel family (of probability measures) if $x\in \overline{\Omega}\longmapsto \eta_x(B)\in \mathbb{R}$ is Borel for any Borel set $B\subset \Gamma$.} of probabilities $\{\eta_x\}_{x\in\overline{\Omega}}$ on $\Gamma$ which disintegrates $\eta$ in the sense that
 	 	\begin{equation}\label{dise}
 	 	\begin{cases}
 	 	\eta(d\gamma)=\int_{\overline{\Omega}} \eta_x(d\gamma) m_0(\,dx),\\
 	 	supp(\eta_x)\subset \Gamma[x] \ \ m_0-\mbox{a.e.} \ x\in \overline{\Omega}
 	 	\end{cases}
 	 	\end{equation}
 	 	(see, e.g., \cite[ Theorem 5.3.1]{5}). Proceeding as in \cite{cc}, we introduce the set-valued map 
 	 	$$
 	 	E:\mathcal{P}_{m_0}(\Gamma)\rightrightarrows \mathcal{P}_{m_0}(\Gamma),
 	 	$$ 
 	 by defining, for any $\eta\in \mathcal{P}_{m_0}(\Gamma)$,
 	 	\begin{equation}\label{ein}
 	 	E(\eta)=\Big\{ \widehat{\eta}\in\mathcal{P}_{m_0}(\Gamma): supp(\widehat{\eta}_x)\subseteq \Gamma^\eta[x] \ \ m_0-\mbox{a.e.} \ x \in \overline{\Omega}\Big\}.
 	 	\end{equation}
We recall that, by \cite[Lemma 3.6]{cc}, the map $E$ has closed graph.\\
Now, we consider the restriction $E_0$ of $E$ to $\lo$, i.e.,
$$
E_0:\mathcal{P}_{m_0}^{\L}(\Gamma_{L_0}) \rightrightarrows \mathcal{P}_{m_0}(\Gamma), \ \ \ E_0(\eta)=E(\eta) \ \ \forall \eta \in\mathcal{P}_{m_0}^{\L}(\Gamma_{L_0}).
$$
We will show that the set-valued map $E_0$ has a fixed point, i.e., there exists $\eta\in \mathcal{P}_{m_0}^{\L}(\Gamma_{L_0})$ such that $\eta\in E_0(\eta)$.
 By \cite[Lemma 3.5]{cc} we have that for any $\eta\in\mathcal{P}_{m_0}^{\L}(\Gamma_{L_0})$, $E_0(\eta)$ is a nonempty convex set. Moreover, we have that
\begin{equation}\label{lin}
	E_0(\mathcal{P}_{m_0}^{\L}(\Gamma_{L_0}))\subseteq \mathcal{P}^{\L}_{m_0}(\Gamma_{L_0}).
 	\end{equation} 
Indeed, let $\eta\in \mathcal{P}_{m_0}^{\L}(\Gamma_{L_0})$ and $\hat{\eta}\in E_0(\eta)$. Since, by Proposition \ref{pr1} one has that 
$$
\Gamma^\eta[x]\subset \Gamma_{L_0} \ \ \ \forall x \in \overline{\Omega},
$$
and by definition of $E_0$ we deduce that
\begin{equation*}
supp(\widehat{\eta})\subset \Gamma_{L_0}.
\end{equation*}
So, $\widehat{\eta}\in\mathcal{P}_{m_0}(\Gamma_{L_0})$. By Lemma \ref{lemmaut}, $\widehat{\eta}\in \mathcal{P}_{m_0}^{\L}(\Gamma_{L_0})$.\\
Since $E$ has closed graph, by Lemma \ref{lemmaut} and \eqref{lin} we have that $E_0$ has closed graph as well. 
Then, the assumptions of Kakutani's Theorem \cite{ka} are satisfied and so, there exists  $\overline \eta\in \mathcal{P}^{\L}_{m_0}(\Gamma_{L_0})$ such that  $\overline \eta\in E_0(\overline \eta)$.
\end{proof}
\noindent
 We recall the definition of a mild solution of the constrained MFG problem, given in \cite{cc}.
 \begin{definition}
 	We say that $(u,m)\in  C([0,T]\times \overline{\Omega})\times C([0,T];\mathcal{P}(\overline{\Omega}))$ is a mild solution of the constrained MFG problem in $\overline{\Omega}$ if there exists a constrained MFG equilibrium $\eta\in\mathcal{P}_{m_0}(\Gamma)$ such that 
 	\begin{enumerate}
 		\item [(i)] $ m(t)= e_t\sharp \eta$ for all $t\in[0,T]$;
 		\item[(ii)] $u$ is given by
 		\begin{equation}\label{v}
 		u(t,x)=\inf_{\tiny\begin{array}{c}
 			\gamma\in \Gamma\\
 			\gamma(t)=x
 			\end{array}} 
 		\left\{\int_t^T \left[L(\gamma(s),\dot \gamma(s))+ F(\gamma(s), m(s))\right]\ ds + G(\gamma(T),m(T))\right\},  
 		\end{equation} 
 		for $(t,x)\in [0,T]\times \overline{\Omega}$.
 	\end{enumerate}
 \end{definition}
 \noindent
 \begin{theorem}\label{cesistenza}
 	Let $\Omega$ be a bounded open subset of $\mathbb{R}^n$ with $C^2$ boundary. Suppose that (L0),(L1), (D1) and (D2) hold true. There exists at least one mild solution $(u,m)$ of the constrained MFG problem in $\overline{\Omega}$. Moreover,
 \begin{enumerate}
 		\item[(i)] $u$ is Lipschitz continuous in $(0,T)\times\overline{\Omega}$;
 		\item[(ii)] $m\in\L(0,T;\mathcal{P}(\overline{\Omega}))$ and $\L(m)=L_0$, where $L_0$ is the constant in \eqref{l0}.
 	\end{enumerate}
 \end{theorem}
 \begin{proof}
 Let $m_0\in\mathcal{P}(\overline{\Omega})$ and let $\eta\in \mathcal{P}_{m_0}^\L(\Gamma)$ be a constrained MFG equilibrium for $m_0$. 
 Then, by Theorem \ref{teoresist} there exists at least one mild solution $(u,m)$ of the constrained MFG problem in $\overline{\Omega}$. Moreover, by Theorem \ref{teoresist} one has that $m\in\L(0,T;\mathcal{P}(\overline{\Omega}))$ and $\L(m)=L_0$, where $L_0$ is the constant in \eqref{l0}. Finally, by Proposition \ref{fw} we conclude that $u$ is Lipschitz continuous in $(0,T)\times \overline{\Omega}$.
 \end{proof}
 \begin{remark}
 Recall that $F:U\times \mathcal{P}(\overline{\Omega})\rightarrow \mathbb{R}$ is strictly monotone if 
 \begin{equation}\label{f1}
 \int_{\overline{\Omega}} (F(x,m_1)-F(x,m_2))d(m_1-m_2)(x)\ \geq\ 0,
 \end{equation}
 for any $m_1,m_2\in {\mathcal P}(\overline \Omega)$, and $\int_{\overline{\Omega}} (F(x,m_1)-F(x,m_2))d(m_1-m_2)(x)=0$ if and only if $F(x,m_1)=F(x,m_2)$ for all $x\in \overline{\Omega}$.\\
 Suppose that $F$ and $G$ satisfy \eqref{f1}. Let $\eta_1$, $\eta_2\in \mathcal{P}_{m_0}^\L(\Gamma)$ be constrained MFG equilibria and let $J_{\eta_1}$ and $J_{\eta_2}$ be the associated functionals, respectively.
 Then $J_{\eta_1}$ is equal to $J_{\eta_2}$. Consequently, if $(u_1,m_1)$, $(u_2,m_2)$ are mild solutions of the constrained MFG problem in $\overline{\Omega}$, then $u_1=u_2$ (see \cite{cc} for a proof).
 \end{remark}
\section{Appendix}
In this Appendix we prove Lemma \ref{lemmaapp}. The only case which needs to be analyzed is when $ x\in\partial\Omega$. We recall that $p\in \partial^p d_\Omega( x)$ if and only if there exists $\epsilon>0$ such that
 	 			\begin{equation}\label{p1}
 	 			d_\Omega( y)-d_\Omega( x) -\langle p, y- x\rangle \geq C| y- x|^2,  \ \ \text{for any} \ y \ \text{such that}\ | y- x|\leq \epsilon,
 	 			\end{equation}
for some constant $C\geq 0$.
Let us show that $\partial^p d_\Omega( x)=D\d( x)[0,1]$. By the regularity of $\d$, one has that 
\begin{equation*}
d_\Omega( y)-d_\Omega( x)-\langle  D\d( x), y- x\rangle\geq \d( y)-\d( x)-\langle  D\d( x), y- x\rangle \geq C | y- x|^2.
\end{equation*}
This shows that $ D\d( x)\in \partial^p d_\Omega( x)$. Moreover, since
\begin{equation*}
d_\Omega( y)-d_\Omega( x)-\langle \lambda  D \d( x), y- x\rangle\geq \lambda\left( d_\Omega( y)-d_\Omega( x)-\langle  D \d( x),  y- x\rangle\right) \ \ \ \forall \lambda \in[0,1],
\end{equation*}
we further obtain the inclusion
\begin{equation*}
D\d( x)[0,1]\subset\partial d_\Omega( x).
\end{equation*}
Next, in order to show the reverse inclusion, let $p\in\partial^p d_\Omega( x)\setminus\{0\}$ and let $ y\in\Omega^c$. Then, we can rewrite \eqref{p1} as
 	 			\begin{equation}\label{p11}
 	 			\d( y)-\d( x) -\langle p, y- x\rangle \geq C| y- x|^2, \ \ \ | y- x|\leq \epsilon.
 	 			\end{equation}
 	 			Since $ y\in \Omega^c$, by the regularity of $\d$ one has that
 	 			\begin{equation}\label{p2}
 	 			\d( y)-\d( x)\leq\langle D\d( x), y- x\rangle +C| y- x|^2
 	 			\end{equation}
for some constant $C\in\mathbb{R}$. By \eqref{p11} and \eqref{p2} one has that
 	 			\begin{equation*}
 	 			\left\langle  D\d( x)-p,\frac{ y- x}{| y- x|}\right\rangle\geq C| y- x|.
 	 			\end{equation*}
 	 			Hence, passing to the limit for $ y\rightarrow  x$, we have that
 	 			\begin{equation*}
 	 			\langle D\d( x)-p, v\rangle \geq 0, \ \ \ \ \forall v\in T_{\Omega^c}( x),
 	 			\end{equation*}
 	 			where $T_{\Omega^c}( x)$ is the contingent cone to $\Omega^c$ at $ x$ ( see e.g. \cite{3v} for a definition).
 	 			Therefore, by the regularity of $\partial\Omega$,
 	 			\begin{equation*}
 	 			 D\d( x)-p= \lambda v( x),
 	 			\end{equation*}
 	 			where $\lambda\geq 0$ and $v( x)$ is the exterior unit normal vector to $\partial\Omega$ in $ x$. Since $v( x)= D\d( x)$, we have that
 	 			\begin{equation*}
 	 			p=(1-\lambda) D\d( x).
 	 			\end{equation*}
 	 			Now, we prove that $\lambda \leq 1$. Suppose that $ y\in \Omega$, then, by \eqref{p1} one has that
 	 			\begin{equation*}
 	 			0=d_\Omega( y)\geq (1-\lambda)\langle D \d( x), y- x\rangle + C| y- x|^2.
 	 			\end{equation*}
 	 			Hence,
 	 			\begin{equation*}
 	 			(1-\lambda)\left\langle D\d( x),\frac{ y- x}{| y- x|}\right\rangle\leq -C| y- x|.
 	 			\end{equation*}
 	 			Passing to the limit for $ y\rightarrow  x$, we obtain
 	 			\begin{equation*}
 	 			(1-\lambda)\left\langle D \d( x), w\right\rangle \leq 0, \ \ \ \ \ \forall w\in T_{\overline{\Omega}}( x),
 	 			\end{equation*}
 	 			where $T_{\overline{\Omega}}( x)$ is the contingent cone to $\Omega$ at $ x$. We now claim that $\lambda\leq 1$. If $\lambda >1$, then $\langle  D \d( x), w \rangle \geq 0$ for all $w\in T_{\overline{\Omega}}( x)$ but this is impossible since $ D\d( x)$ is the exterior unit normal vector to $\partial\Omega$ in $ x$.\\
 	 			Using the regularity of $\d$, simple limit-taking procedures permit us to prove that $\partial d_\Omega( x)=D\d( x)[0,1]$ when $ x\in\partial \Omega$. This completes the proof of Lemma \ref{lemmaapp}.


\begin{thebibliography}{abc99xyz} 
	\bibitem{adams}
	Adams, R. A., \textsl{Sobolev Spaces}, Academic Press, New York, 1975.
 	 	\bibitem{5}
 	 	Ambrosio, L., Gigli, N., Savare, G., \textsl{Gradient flows in metric spaces and in the space of probability measures. Second edition,} Lectures in Mathematics
 	 	ETH Z\"urich, Birkh\"auser Verlag, Basel, 2008.
\bibitem{aa}
Arutyanov, A. V., Aseev, S. M., \textsl{Investigation of the degeneracy phenomenon of the maximum principle for optimal control problems with state constraints}, SIAM J. Control Optim. 35, 930-952, 1997. 
\bibitem{bb}
Benamou, J. D., Brenier, Y., \textsl{A computational fluid mechanics solution to the Monge-Kantorovich mass transfer problem}, Numer. Math., 84, 375-393, 2000.

\bibitem{bc}
Benamou, J. D., Carlier, G., \textsl{Augmented Lagrangian Methods for Trasport Optimization, Mean Field Games and Degenerate Elliptic Equations}, J. Opt. Theor. Appl., 167, No. 1, 1-26, 2015.
\bibitem{bcs}
Benamou, J. D., Carlier, G., Santambrogio, F., \textsl{Variational Mean Field Games}, In: Bellomo N., Degond P., Tadmor E. (eds) Active Particles, Vol 1, Modeling and Simulation in Science, Engineering and Technology. Birkh\"auser, 141-171, 2017.
\bibitem{b}
Brenier, Y., \textsl{Minimal geodesics on groups of volume-preserving maps and generalized solutions of the Euler equations}, Comm. Pure Appl. Math., 52, No. 4, 411-452, 1999.
\bibitem{bettiol2007normality}
Bettiol, P., Frankowska, H.,
\textsl{Normality of the maximum principle for nonconvex constrained bolza
	problems}, Journal of Differential Equations, 243(2):256--269, 2007.
\bibitem{bettiol2008holder}
Bettiol, P., and Frankowska, H.,
\textsl{H{\"o}lder continuity of adjoint states and optimal controls for
	state constrained problems}, Applied Mathematics and Optimization, 57(1):125--147, 2008.
\bibitem{bettiol2016normality}
Bettiol, P., Khalil, N., and Vinter, R. B.,
\textsl{Normality of generalized euler-lagrange conditions for state
	constrained optimal control problems}, Journal of Convex Analysis, 23(1):291--311, 2016.
\bibitem{cc}
Cannarsa, P., Capuani, R., \textsl{Existence and uniqueness for Mean Field Games with state constraints}, http://arxiv.org/abs/1711.01063.
\bibitem{ccc} 
Cannarsa, P., Castelpietra, M., and Cardaliaguet, P., \textsl{Regularity properties of a attainable sets under state constraints}, 120-135, Series on Advances in Mathematics for Applied Sciences, Vol 76, World Sci. Publ., Hackensack, NJ, 2008.
\bibitem{pc}
Cardaliaguet, P., \textsl{Weak solutions for first order mean field games with local coupling. In Analysis and geometry in
control theory and its applications}, volume 11 of Springer INdAM Ser., pages 111-158, Springer, Cham, 2015.
\bibitem{cms}
Cardaliaguet, P., Mészáros, A. R., Santambrogio, F., \textsl{First order mean field games with density constraints:pressure equals price}, SIAM J. Control Optim., 54(5):2672-2709, 2016.
\bibitem{c}
Cesari, L., \textsl{Optimization-Theory and Applications. Problems with Ordinary Differential Equations}, Applications of Mathematics, Vol 17, Springer-Verlag, New York, 1983.
\bibitem{clarke}
Clarke, F. H., \textsl{Optimization and Nonsmooth Analisis}, John Wiley \& Sons, New York, 1983.
\bibitem{dubovitskii1964extremum}
Dubovitskii, A. Y., and Milyutin, A. A., \textsl{Extremum problems with certain constraints}, Dokl. Akad. Nauk SSSR, Vol 149, 759--762, 1964.
\bibitem{frankowska2006regularity}
Frankowska, H.,
\textsl{Regularity of minimizers and of adjoint states in optimal control
	under state constraints}, Journal of Convex Analysis, 13(2):299, 2006.

\bibitem{frankowska2009normality}
Frankowska, H.,
\textsl{Normality of the maximum principle for absolutely continuous
	solutions to Bolza problems under state constraints}, Control and Cybernetics, 38:1327--1340, 2009.
\bibitem{fran}
Frankowska, H., \textsl{Optimal control under state constraints},
In Proceedings of the International Congress of Mathematicians
2010 (ICM 2010) (In 4 Volumes) Vol. I: Plenary Lectures and Ceremonies Vols.
II--IV: Invited Lectures, 2915--2942. World Scientific, 2010.
\bibitem{galbraith2003lipschitz}
Galbraith, G. N., and Vinter, R. B., \textsl{Lipschitz continuity of optimal controls for state constrained problems}, SIAM journal on control and optimization, 42(5):1727--1744, 2003.
\bibitem{hager1979lipschitz}
Hager, W. W., \textsl{Lipschitz continuity for constrained processes}, SIAM Journal on Control and Optimization, 17(3):321--338, 1979.
 	 	\bibitem{h1}
 	 	Huang, M., Caines, P.E., and Malham\'{e}, R.P., \textsl{Large-population cost-coupled LQG problems with nonuniform agents: Individual-mass behavior and decentralized $\epsilon$-Nash equilibria},
 	 	Automatic Control, IEEE Transactions on 52, no. 9, 1560-1571, 2007.
 	 	\bibitem{h2}
 	 	Huang, M., Malham\'{e}, R.P., Caines, P.E., \textsl{Large population stochastic dynamic games: closed-loop McKean-Vlasov systems and the Nash certainly equivalence principle,}
 	 	Communication in information and systems Vol 6, no. 3, pp. 221-252, 2006.
 	 	\bibitem{8}
 	 	Lasry, J.-M., Lions, P.-L., \textsl{Jeux \`{a} champ moyen. I. Le cas stationnaire}, C. R. Math. Acad. Sci. Paris 343, no. 9, 619-625, 2006.
 	 	%
 	 	\bibitem{9}
 	 	Lasry, J.-M., Lions, P.-L., \textsl{Jeux \`{a} champ moyen. II. Horizon fini et contr\^{o}le optimal}, C. R. Math. Acad. Sci. Paris 343, no. 9, 679-684, 2006.
 	 	\bibitem{10}
 	 	Lasry, J.-M., Lions, P.-L., \textsl{Mean field games}, Jpn. J. Math. 2, no.1, 229-260, 2007.
 	 	\bibitem{lr}
 	 	Loewen, P., Rockafellar, R. T., \textsl{The adjoint arc in nonsmooth optimization}, Trans. Am. Math. Soc. 325, 39-72, 1991.
 	 	\bibitem{ka}
 	 	Kakutani, S., \textsl{A generalization of Brouwer's fixed point theorem,} Duke Math J., Vol 8, no. 3, 457-459, 1941.
 	 	\bibitem{malanowski1978regularity}
 	 	Malanowski, K., \textsl{On regularity of solutions to optimal control problems for systems with control appearing linearly}, Archiwum Automatyki i Telemechaniki, 23(3):227--242, 1978.
 	 	\bibitem{milyutin2000certain}
 	 	Milyutin, A. A., \textsl{On a certain family of optimal control problems with phase constraint}, Journal of Mathematical Sciences, 100(5):2564--2571, 2000.
 	 	\bibitem{rv}
 	 	Rampazzo, F., Vinter, R. B., \textsl{Degenerate optimal control problems with state constraints}, SIAM J. Control Optim. 39, 989-1007, 2000.
 	 	\bibitem{3v}	
 	 	Vinter, R. B., \textsl{Optimal control}, Birkh\"auser Boston, Basel, Berlin, 2000.
 	 	
 	 \end{thebibliography}
 	 \end{document}